\documentclass[reqno,11pt]{amsart}
\usepackage{amssymb}
\usepackage{mathrsfs}
\usepackage{txfonts}
\usepackage{amsfonts}
\usepackage{amstext}
\usepackage{amssymb}
\usepackage{amsmath}
\usepackage{graphicx}
\usepackage{hyperref}
\usepackage{url}
\usepackage{tikz}
\usepackage{xparse}

\makeatletter
\def\sim@x@scale{.15}
\def\sim@y@scale{.05}
\def\sim@y@thick{.02}
\newsavebox\sim@upper
\newsavebox\sim@lower
\NewDocumentCommand{\xSim}{ O{} m }{\TextOrMath{\PackageError{TEST}{`\string\xSim` is valid in math mode only.}{}}{
\sbox\sim@upper{$\scriptsize #2$}
\sbox\sim@lower{$\scriptsize #1$}
\pgfmathparse{min(max(\wd\sim@upper/1em, \wd\sim@lower/1em, 1.0), 1.5)}
\edef\sim@ratio{\pgfmathresult}
\def\sim@x {\sim@x@scale * \sim@ratio}
\def\sim@y {\sim@y@scale * \sim@ratio}
\def\sim@@y{\sim@y@thick * \sim@ratio}
\pgfmathparse{floor(max(\wd\sim@upper/1em, \wd\sim@lower/1em)) + 1}
\edef\sim@wd{\pgfmathresult em}
\mathrel{
\begin{tikzpicture}[baseline=-.7ex]
\filldraw[line width=.2pt] 
(0, 0)
.. controls +(\sim@x, \sim@y+\sim@@y) and +(-\sim@x, -\sim@y) .. 
+(\sim@wd, 0) 
node[midway, above] {\usebox\sim@upper} 
node[midway, below] {\usebox\sim@lower}
.. controls +(-\sim@x, -\sim@y-\sim@@y) and +(\sim@x, \sim@y) .. 
(0, 0);
\end{tikzpicture}
}
}}
\makeatother
\hypersetup{
    colorlinks  = true,
    citecolor  = blue,
    linkcolor  = blue,
    urlcolor   = blue
}
\allowdisplaybreaks
\newtheorem{theorem}{Theorem}[section]
\newtheorem{proposition}[theorem]{Proposition}
\newtheorem{lemma}[theorem]{Lemma}
\newtheorem{corollary}[theorem]{Corollary}
\newtheorem{remark}[theorem]{Remark}

\newtheorem{definition}[theorem]{Definition}

\newtheorem{example}[theorem]{Example}

\renewcommand{\theequation}{\thesection.\arabic{equation}}
\newenvironment{notation}{\smallskip{\sc Notation.}\rm}{\smallskip}
\newenvironment{acknowledgement}{\smallskip{\sc Acknowledgement.}\rm}{\smallskip}

\DeclareMathOperator*{\esup}{esup}

\numberwithin{equation}{section}
\newcounter{counterConstant}

\makeatletter

\newcommand{\Rmnum}[1]{\expandafter\@slowromancap\romannumeral #1@}
\@namedef{subjclassname@2020}{\textup{2020} Mathematics Subject Classification}
\makeatother
\input{tcilatex}

\makeatletter
\@namedef{subjclassname@2020}{%
  \textup{2020} Mathematics Subject Classification}
\makeatother

\def\Xint#1{\mathchoice
{\XXint\displaystyle\textstyle{#1}}%
{\XXint\textstyle\scriptstyle{#1}}%
{\XXint\scriptstyle\scriptscriptstyle{#1}}%
{\XXint\scriptscriptstyle\scriptscriptstyle{#1}}%
\!\int}
\def\XXint#1#2#3{{\setbox0=\hbox{$#1{#2#3}{\int}$ }
\vcenter{\hbox{$#2#3$ }}\kern-.6\wd0}}

\def\oint{\Xint-}

\def\FRAME#1#2#3#4#5#6#7#8
{
 \begin{figure}[H]
 \begin{center}
 \includegraphics[height=#3]{#7}
 \caption{#5}
 \label{#6}
 \end{center}
 \end{figure}
}

\def\enddoc{\end{document}}

\begin{document}
\title[On the well-posedness of PMEs on general metric measure spaces]{On
the well-posedness of porous medium equations on general metric measure
spaces}
\author[Chang]{Diwen Chang}
\address{School of Mathematics and Statistics, Beijing Technology and
Business University, Beijing, 102488, China.}
\email{dwchang@btbu.edu.cn}
\thanks{Diwen Chang is supported by the Research Foundation for Youth Scholars of Beijing Technology and Business University, China \textbf{(No. RFYS2025)}.}

\begin{abstract}
	On general metric measure spaces, we develop a new well-posedness theory for the signed porous medium equation and its fast diffusion counterpart 
	\[
	\partial_t u = \mathcal{L}\left(|u|^{m-1}u\right), \qquad m>0,
	\]
	where $\mathcal{L}$ is the associated non-positive self-adjoint operator of a symmetric Dirichlet form. The theory does not rely on a Gelfand triple or compact embeddings; instead, it is built upon the extended Dirichlet space $\mathcal{F}_e$ and auxiliary spaces $V^q:=L^q\cap\mathcal{F}_e$, whose uniform convexity plays a key role in the proof. The proof uses only the existence of the Dirichlet form and its extension; no additional regularity of the form or geometric assumptions on the underlying space are needed. Consequently, the results apply to a wide range of metric measure spaces, including non-smooth fractals.
\end{abstract}



\maketitle
\tableofcontents

\section{Introduction}
\label{sec:Intro} 

Let $(M,d)$ be a locally compact, separable metric space equipped with a Radon measure $\mu$ of full support (i.e., $\operatorname{supp}\mu=M$). 
The triple $(M,d,\mu)$ is called a \emph{metric measure space}.

In this article, we study the following quasi-linear evolution equation
\begin{equation}
	\partial_t u = \mathcal{L}\left( |u|^{m-1}u \right), \qquad m>0, \label{eq:1}
\end{equation}
and the well-posedness of associated Cauchy problem
\begin{equation}
	\left\{
	\begin{aligned}
		&\partial_t u = \mathcal{L}\left( |u|^{m-1}u \right), \quad &&(t,x)\in(0,T)\times M,\\
		&u(0,\cdot) = u_0, \quad &&x\in M,
	\end{aligned}
	\right. \label{eq:cauchy}
\end{equation}
on a metric measure space $(M,d,\mu)$, where $u_{0}$ is a measurable function, and $\mathcal{L}$ is a non-positive self-adjoint operator associated with a Dirichlet form $(\mathcal{E},\mathcal{F})$ on $L^{2}(M,\mu)$ (cf. \cite{FukushimaOshimaTakeda.2011.489}). Concrete realizations of $\mathcal{L}$ include the Laplace--Beltrami operator on Riemannian manifolds and the fractional Laplacian on Euclidean spaces. 

Equation \eqref{eq:1}—known as the porous medium equation (PME) when $m>1$, and the fast diffusion equation (FDE) when $0<m<1$ (including their fractional counterparts)—has been extensively studied in Euclidean domains and on Riemannian manifolds. Yet, on non-smooth metric measure spaces such as fractals, a well-posedness theory without imposing \(L^2\)-regularity on the nonlinearity remains unavailable. This is precisely the primary motivation for the present work.


The literature on the PME and FDE is vast. In Euclidean domains, the PME and FDE have been systematically treated in the monograph \cite{VazquezPMEBook} and the references therein; their fractional counterparts were introduced in the pioneering works \cite{PabloQRVazquez2011Adv,PabloQRVazquez2012CPAM}. On Riemannian manifolds, geometric properties such as curvature play a prominent role. On manifolds with non-positive sectional curvature, such as hyperbolic spaces and Cartan--Hadamard manifolds, well-posedness and asymptotic behaviour for broad classes of initial data have been addressed in a substantial body of work (see, e.g., \cite{BonforteGrilloVazquez2008,GrilloMuratori2014,GrilloMuratori2016,GrilloMuratoriPunzo2018,GrilloMuratoriVazquez2017Adv,GrilloMuratoriVazquez2019,Vazquez2015JMPA}). The fractional PME on manifolds was recently studied via a potential method tailored to non-positively curved settings \cite{BerchioBonforteGrilloMuratori2024}. This strategy was later adapted to the PME on manifolds with nonnegative Ricci curvature, employing instead a Green function technique \cite{GrilloMonticelliPunzo2025JDE}. More recently, the well-posedness of the Cauchy problem for the Leibenson equation (which reduces to the PME or FDE when $p=2$) was established on arbitrary geodesically complete Riemannian manifolds, without imposing any curvature restriction \cite{Suerig2026}.

The above discussion concerns smooth Riemannian manifolds. On non-smooth metric measure spaces such as fractals, however, the very definition of a Laplacian is already a non-trivial first step, since these spaces lack a differential structure. Analysis on fractals originated with the probabilistic construction of Brownian motion on the Sierpi\'{n}ski gasket \cite{BarlowPerkins.1988.PTRF543} and the Sierpi\'{n}ski carpet \cite{BarlowBass.1989.AIHPPS225}. Subsequently, based on graph approximations, an analytic framework for defining a Laplacian on a wide class of fractals was developed \cite{Kigami.1993.TAMS721}. All these constructions---whether probabilistic or analytic---are unified within the theory of Dirichlet forms \cite{ChenFukushima.2012.479,FukushimaOshimaTakeda.2011.489}. In this framework, a given Dirichlet form is associated with a unique non-positive self-adjoint operator, which serves precisely as the Laplacian on the fractal.

Once the Laplacian is available, a wide range of PDEs has been studied on fractals and general metric measure spaces, including the heat equation and heat kernel estimates \cite{BarlowGrigoryanKumagai.2012.JMSJ1091,GrigoryanHuHu.2018.AM433,GrigoryanHu.2008.IM81,GrigoryanHu.2014.MMJ505,GrigoryanHuLau.2015.JMSJ1485,GrigoryanTelcs.2012.AP1212,Kumagai.1993.PTRF205}, the wave equation \cite{DalrympleStrichartzVinson.1999.JFAA203}, and semilinear PDEs \cite{Falconer.1999.CMP235,FalconerHu.1999.JMAA552,FalconerHu.2001.JMAA606,FalconerHuSun2012}. These works have profoundly revealed how fractal geometry, in particular fractal dimensions, affects the behaviour of solutions. A separate line of research, based on the abstract first-order calculus for Dirichlet forms developed in \cite{HinzRocknerTeplyaev2013}, has initiated the study of first-order PDEs—such as the viscous Burgers equation \cite{HinzMeinert2020} and the continuity/transport equation \cite{HinzSchefer2024}—on fractals.

Nevertheless, research on equation \eqref{eq:1} on general metric measure spaces remains limited, largely due to the inherent difficulty arising from the coupling between the power-law nonlinearity \(\Psi(u)=|u|^{m-1}u\) and the abstract operator \(\mathcal{L}\). In the stochastic setting, the works \cite{RocknerWuXie2018, RocknerWuXie2024} established a well-posedness theory for the more general equation \(dX(t)- \mathcal{L}(P(X(t)))dt\ni B(t,X(t))dW(t)\), where $B$ is a Hilbert-Schmidt operator-valued map and $W$ is a Wiener process. Their framework, however, requires the underlying Dirichlet form to be transient, and relies essentially on stochastic tools such as the \(L^p\)-It\^{o} formula, which impose additional regularity not needed in the purely deterministic setting.

On the deterministic side, the closest result is in \cite[Section 3]{AmbrosioMondinoSavare2019}, which also treats \(\partial_t u = \mathcal{L}(P(u))\) within a Gelfand triple framework. While general in scope, it is not tailored to the power case: \cite[Theorem 3.4]{AmbrosioMondinoSavare2019} requires \(P'\) to take values in a closed subinterval of \((0,\infty)\), which fails for \(\Psi(u)=|u|^{m-1}u\); and although \cite[Theorem 3.7]{AmbrosioMondinoSavare2019} does cover the power case, it does so only under finite total measure and \(L^\infty\) initial data. More fundamentally, the Gelfand triple structure forces \(P(u)\in L^2\), an \(L^2\)-regularity that is not required by the equation itself—the weak formulation of \eqref{eq:cauchy} only needs finiteness of the energy of \(\Psi(u)\). This is not a limitation of \cite{AmbrosioMondinoSavare2019} per se, but rather a natural consequence of the broader scope for which their framework is designed. For the specific power-law equation \eqref{eq:cauchy}, however, such a structure imposes unnecessary restrictions.

In view of the above, we develop a new well-posedness theory for \eqref{eq:cauchy} in which the function spaces are dictated by the natural energy estimates of the equation, namely $u\in L^{m+1}$ and the finiteness of the energy of $|u|^{m-1}u$, rather than by a pre-fixed Hilbert structure. 

Our main novelties are twofold. 
\begin{itemize}
	\item First, to avoid the structural constraints of the Gelfand triple—which inherently imposes an \(L^2\)-regularity on the nonlinearity \(P(u)\) that is not required by the equation—we replace the original domain \(\mathcal{F}\) by the extended Dirichlet space \(\mathcal{F}_e\). This allows us to introduce the spaces \(V^q := L^q \cap \mathcal{F}_e\), where the nonlinearity \(\Psi(u) = |u|^{m-1}u\) is naturally placed in \(\mathcal{F}_e\) with finite energy, but not necessarily in \(L^2\). These spaces are tailored precisely to the natural energy estimates \(u\in L^{m+1}\) and \(\Psi(u)\in\mathcal{F}_e\).
	
	\item Second, the use of \(\mathcal{F}_e\) comes at a cost: since \((\mathcal{F}_e,\sqrt{\mathcal{E}})\) need not be complete (cf. \S~\ref{sub:Fe}), reflexivity is not automatic for the intersection spaces \(V^q\). We prove that \((V^q,\|\cdot\|_{V^q})\) is complete and uniformly convex (hence reflexive) for every \(1<q<\infty\) (cf. Proposition~\ref{prop:Vq_uniformly_convex}). This non-trivial result is a key contribution of this work: it provides the reflexive Banach setting necessary for the application of the Minty--Browder theorem, entirely without relying on a Gelfand triple.
\end{itemize}

We now introduce the functional setting for our weak solutions. Set $
V := V^{1+1/m} = L^{1+1/m}(M,\mu) \cap \mathcal{F}_e$, and denote by \(V^*\) the dual space of $V$.

\begin{definition}[Weak solution of \eqref{eq:cauchy}]
	\label{def:weaksol}
	Let \(T>0\) and \(u_0\in L^{m+1}(M,\mu)\). A function
	\[
	u\in L^\infty(0,T;L^{m+1}(M,\mu)) \cap W^{1,2}(0,T;V^*)
	\]
	is called a \emph{weak solution} of the Cauchy problem \eqref{eq:cauchy} on \([0,T)\), if \(\Psi(u):=|u|^{m-1}u\in L^2(0,T;V)\) and
	\[
	-\int_0^T\int_M u\,\partial_t\varphi\,d\mu\,dt
	+ \int_0^T \mathcal{E}(\Psi(u)(t),\varphi(t))\,dt
	= \int_M u_0\,\varphi(0)\,d\mu
	\]
	for every test function \(\varphi\in C_c^1([0,T);V)\).
\end{definition}

\begin{remark}
	The weak formulation involves only the energy form \(\mathcal{E}\), not the pointwise action of \(\mathcal{L}\). Since \(\Psi(u)\) need not be \(L^2\)-integrable, we work in the extended Dirichlet space \(\mathcal{F}_e\), where \(\mathcal{E}\) is well-defined without requiring \(L^2\)-regularity. This is precisely why the space \(V\) is chosen as \(L^{1+1/m}\cap\mathcal{F}_e\): it matches the natural energy estimates \(u\in L^{m+1}\) and \(\Psi(u)\in\mathcal{F}_e\). The precise definitions of \(\mathcal{F}\) and \(\mathcal{F}_e\) are given in Section~\ref{Sec:DF}.
\end{remark}

Our main result is the following theorem.

\begin{theorem}
	\label{thm:main}
	Let \((M,d,\mu)\) be a metric measure space and let \((\mathcal{E},\mathcal{F})\) be a Dirichlet form on \(L^2(M,\mu)\) with associated non-positive self-adjoint operator \(\mathcal{L}\) (cf.\eqref{eq:Laplacian}). For every \(T>0\) and every \(u_0\in L^{m+1}(M,\mu)\), the following hold:
	
	\begin{enumerate}
		\item[(i)] (Existence and uniqueness.) There exists a unique weak solution \(u\) of \eqref{eq:cauchy} on $[0,T)$ in the sense of Definition~\ref{def:weaksol}, satisfying the energy estimates
		\begin{equation}\label{eq:bound}
			\|u\|_{L^\infty(0,T;L^{m+1})} \leq \|u_0\|_{L^{m+1}},
		\end{equation}
		and
		\begin{equation}\label{eq:boundv}
			\|\Psi(u)\|_{L^2(0,T;V)}
			\leq \sqrt{T\|u_0\|_{L^{m+1}}^{2m} + \frac{\|u_0\|_{L^{m+1}}^{m+1}}{m+1}}.
		\end{equation}
		
		\item[(ii)] (Comparison principle.) If \(\hat{u}_0,\tilde{u}_0\in L^{m+1}(M,\mu)\) satisfy \(\hat{u}_0\le \tilde{u}_0\) \(\mu\)-a.e., then the corresponding solutions satisfy \(\hat{u}(t,\cdot)\le \tilde{u}(t,\cdot)\) \(\mu\)-a.e. for almost every \(t\in(0,T)\). In particular, if \(u_0\ge 0\) \(\mu\)-a.e., then \(u(t,\cdot)\ge 0\) \(\mu\)-a.e. for almost every \(t\in(0,T)\).
		
		\item[(iii)] (Stability with respect to initial data.) Suppose that, there is a sequence \(\{u_{0,k}\}_{k\geq1}\subset L^{m+1}(M,\mu)\), satisfying
		\[
		\lim_{k\to\infty}\|u_{0,k}-u_0\|_{L^{m+1}(M,\mu)}=0.
		\]
		 For every $k\geq 1$, let $u^{(k)}$ be the weak solution of \eqref{eq:cauchy} on $[0,T)$ with initial datum $u_{0,k}$. Then, the corresponding solutions \(u^{(k)}\) satisfy
		\[
		\lim_{k\to\infty}\|u^{(k)}-u\|_{L^2(0,T;L^{m+1}(M,\mu))}=0.
		\]
	\end{enumerate}
\end{theorem}
\begin{remark}
	\label{rem:infinite}
	The notion of weak solution extends naturally to the infinite time horizon \(T=\infty\). In this case, Definition~\ref{def:weaksol} is modified by replacing \(u\in W^{1,2}(0,T;V^*),\Psi(u)\in L^2(0,T;V)\) with \(u\in  W^{1,2}_{\mathrm{loc}}([0,\infty);V^*), \Psi(u)\in L^2_{\mathrm{loc}}(0,T;V)\), respectively.
	
	With this extension, except \eqref{eq:boundv}, the remaining part of Theorem~\ref{thm:main}(i),(ii) also holds for \(T=\infty\). Indeed, for \(0<T_1<T_2\), the restriction of the solution on \([0,T_2)\) to \([0,T_1)\) is, by Theorem~\ref{thm:main}(i), the solution on \([0,T_1)\). For each \(j\in\mathbb{N}\), let \(u^{(j)}\) denote the solution on \([0,j)\). Gluing these solutions by $u_\infty(t)|_{[0,j)}=u^{(j)}(t)$ gives a unique
	\[
	u_\infty\in L^\infty([0,\infty);L^{m+1})\cap W^{1,2}_{\mathrm{loc}}([0,\infty);V^*)
	\]
	satisfying the weak formulation on \([0,\infty)\), since every test function in \(C_c^1([0,\infty);V)\) has support contained in some finite interval \([0,j)\), and the \(L^\infty(0,T;L^{m+1}(M,\mu)) \)-bound follows from \eqref{eq:bound} in Theorem~\ref{thm:main}.
\end{remark}

The proof of existence in Theorem~\ref{thm:main}(i) is based on two classical tools: the Rothe method for time discretization and the Minty--Browder theorem for the elliptic subproblems (see, e.g., \cite{Roubíček2013} and the references therein for the general theory; for recent applications to evolution equations and variational inequalities, see \cite{BartoszChengKalita2015,PengZhaoLong2025,SlodickaVrabel2017, YinZeng2024}). The identification of the weak limit of the nonlinearity is achieved via the classical Minty trick (see, e.g., \cite[(a) in Page 474]{Zeidler1990IIB}), which avoids the need for compact embeddings, as in \cite{SlodickaVrabel2017}. Uniqueness, the comparison principle, and stability are proved in \S~\ref{sub:unicom} and \S~\ref{subsec:stability} using standard monotonicity arguments.

In addition to the novelties discussed above, the present framework offers several distinct advantages.

\begin{itemize}
	\item \textit{Minimal geometric assumptions.} We work under the standard assumptions stated at the outset of this section: local compactness, separability, and a full-support Radon measure. These ensure that the extended Dirichlet space $\mathcal{F}_e$ is well-defined (see \cite[p.~40, the discussion preceding Theorem 1.5.2]{FukushimaOshimaTakeda.2011.489}) and existence of the compact exhaustion of $M$, which is helpful to prove Theorem~\ref{thm:main}(ii). No additional geometric conditions are imposed. In particular, our framework covers both smooth Riemannian manifolds and non-smooth self-similar fractals within a unified theory.
	
	\item \textit{Weak assumptions on the Dirichlet form.} The theory requires only the existence of a Dirichlet form and its extended Dirichlet space; neither regularity nor locality of the form is needed. Consequently, our results apply simultaneously to local operators (e.g., Laplace--Beltrami), nonlocal operators (e.g., fractional Laplacians), and operators on fractal spaces, all within a single functional-analytic framework.
	
	\item \textit{Identification of nonlinear limits without compactness.} In the passage to the limit, we obtain weak limits of both the approximate solution and its nonlinearity. To identify the latter as the nonlinearity of the former, we employ the classical Minty trick, which uses the strict monotonicity of the power nonlinearity. This avoids any compactness assumption in the limiting argument.
\end{itemize}

The rest of this paper is organized as follows. In Section~\ref{Sec:DF}, we recall the basic properties of Dirichlet forms and their associated extended Dirichlet spaces. In Section~\ref{Sec:Vq}, we prove the completeness and reflexivity of $(V^q,\|\cdot\|_{V^q})$ for every $1<q<\infty$, and establish some auxiliary functional inequalities for later use. In Section~\ref{Sec:Epi}, we investigate the existence and uniqueness of the elliptic equation obtained by time discretization. The proof of our main theorem is completed in Section~\ref{Sec:Para} via the Rothe method and the Minty trick. Finally, we illustrate our main results with specific examples in Section~\ref{Sec:Outro}.

\begin{notation}
	We set
	\[
	\Psi(u):=|u|^{m-1}u,\qquad u\in\mathbb{R},
	\]
	and denote by $\Psi^{-1}$ its inverse, given by $\Psi^{-1}(v)=|v|^{1/m}\operatorname{sgn}(v)$ for $v\in\mathbb{R}$. Both $\Psi$ and $\Psi^{-1}$ are continuous, strictly increasing, and satisfy $\Psi(0)=\Psi^{-1}(0)=0$.
	
We write \(a\wedge b:=\min\{a,b\}\) and \(a\vee b:=\max\{a,b\}\) for \(a,b\in\mathbb{R}\). For a function \(f\), we set \(f_+:=f\vee 0\) and \(f_-:=(-f)\vee 0\).
	
	For $1\le q\le\infty$, we abbreviate $L^q(M,\mu)$ by $L^q$, and write $\|\cdot\|_{L^q}$ for its norm. For a general normed space $X$, we denote its norm by $\|\cdot\|_X$, and let $\langle\cdot,\cdot\rangle_X$ denote the duality pairing between $X^*$ and $X$.
	
	For $1\le q\le\infty$, we write $L^q(0,T;X)$ for the Bochner--Lebesgue space of $X$-valued functions on $[0,T]$. We also write $W^{1,2}(0,T;X)$ for the Sobolev--Bochner space of functions whose weak time derivative belongs to $L^2(0,T;X)$. We write $C_c^1([0,T);X)$ for the space of $X$-valued $C^1$-functions compactly supported in $[0,T)$. We follow the standard conventions for Bochner--Lebesgue spaces (see, e.g., \cite[\S 5.9.2]{Evans.2010.749}).
	
     For any symmetric bilinear form $\mathcal{B}$, we write $\mathcal{B}(u):=\mathcal{B}(u,u)$ whenever no confusion can arise.
\end{notation}

\section{Preliminaries on Dirichlet forms}
\label{Sec:DF}

In this section we recall the definition and some basic properties of Dirichlet forms and related extended Dirichlet spaces.

\subsection{Dirichlet forms}

We begin by recalling the definition of a Dirichlet form on $L^2(M,\mu)$.

\begin{definition}[Dirichlet form]
	\label{def:dirichlet}
	A symmetric non-negative bilinear form $(\mathcal{E},\mathcal{F})$ on $L^2(M,\mu)$ is called a \emph{Dirichlet form} if the following conditions hold:
	\begin{enumerate}
		\item $\mathcal{F}$, the domain of $\mathcal{E}$, is dense in $L^2(M,\mu)$;
		\item (Closeness) $\mathcal{F}$ is a Hilbert space with respect to the inner product
		\[
		(u,v)_{\mathcal{E}_1}:=(u,v)_{L^2}+\mathcal{E}(u,v),\qquad u,v\in\mathcal{F},
		\]
		with associated norm $\|u\|_{\mathcal{E}_1}:=\sqrt{(u,u)_{\mathcal{E}_1}}$;
		\item (Markovian property) for every $u\in\mathcal{F}$, the function $\bar u:=(u\vee 0)\wedge 1$ belongs to $\mathcal{F}$ and satisfies
		\[
		\mathcal{E}(\bar u)\le \mathcal{E}(u).
		\]
	\end{enumerate}
\end{definition}

Since a Dirichlet form is a closed symmetric form, the general correspondence between closed symmetric forms and non-positive self-adjoint operators applies: there exists a unique non-positive self-adjoint operator $\mathcal{L}$ on $L^2(M,\mu)$ such that
\begin{equation}\label{eq:Laplacian}
	\mathcal{F}=D\left(\sqrt{-\mathcal{L}}\right),\qquad
	\mathcal{E}(u,v)=\int_M \left(\sqrt{-\mathcal{L}}u\right)\left(\sqrt{-\mathcal{L}}v\right)\,d\mu,\qquad u,v\in\mathcal{F}.
\end{equation}
Moreover, 
\begin{equation*}
	\int_M\left(-\mathcal{L}u\right)v\, d\mu=\mathcal{E}(u,v)\qquad u\in D(\mathcal{L}), v\in\mathcal{F}.
\end{equation*}

For simplicity, we say that $\mathcal{L}$ is the \emph{Laplacian} associated with $(\mathcal{E},\mathcal{F})$. This operator will be central to our study, as the equation $\partial_t u=\mathcal{L}\left(|u|^{m-1}u\right)$ is precisely the evolution equation associated with this associated Laplacian.

Before presenting concrete examples, we recall the notions of regularity and locality for Dirichlet forms. A Dirichlet form $(\mathcal{E},\mathcal{F})$ is called \emph{regular} if $\mathcal{F}\cap C_c(M)$ is dense in $\mathcal{F}$ with respect to the $\mathcal{E}_1$-norm and in $C_c(M)$ with respect to the uniform norm; it is called \emph{local} if $\mathcal{E}(u,v)=0$ whenever $u,v\in\mathcal{F}$ have disjoint compact supports (see, e.g., \cite{FukushimaOshimaTakeda.2011.489}).

We emphasize that neither regularity nor locality is required for our main results; they are mentioned only for completeness in the examples.

The following examples illustrate the scope of the above framework.

\begin{example}
	\label{ex:dirichlet}
	The following three classes of Dirichlet forms are particularly relevant to our framework.
	
	\begin{enumerate}
		\item[(i)] \textit{Euclidean Sobolev spaces.} On $M=\mathbb{R}^n$ with Lebesgue measure, let $\mathcal{F}=W^{1,2}(\mathbb{R}^n)$ and
		\[
		\mathcal{E}(u,v)=\int_{\mathbb{R}^n}\nabla u\cdot\nabla v\,dx.
		\]
		This is a regular local Dirichlet form, and its associated Laplacian is the classical Laplacian \(\mathcal{L}=\Delta\), where \(\Delta:=\sum_{i=1}^n \frac{\partial^2}{\partial x_i^2}\).
		
		\item[(ii)] \textit{Fractional Sobolev spaces.} For $0<\beta<2$, define
		\[
		\mathcal{E}_\beta(u,v)=\int_{\mathbb{R}^n}|\xi|^{\beta}\hat u(\xi)\overline{\hat v(\xi)}\,d\xi
		\]
		for $u,v\in L^2(\mathbb{R}^n)$ with domain $\mathcal{F}_\beta=\{u\in L^2(\mathbb{R}^n):\mathcal{E}_\beta(u,u)<\infty\}$, where $\hat u$ and $\hat v$ are the Fourier transforms of $u$ and $v$, respectively. Equivalently, for $u,v\in\mathcal{F}_\beta$,
		\[
		\mathcal{E}_\beta(u,v)=C_{n,\beta}\int_{\mathbb{R}^n}\int_{\mathbb{R}^n}\frac{(u(x)-u(y))(v(x)-v(y))}{|x-y|^{n+\beta}}\,dx\,dy,
		\]
		with
		\[
		C_{n,\beta}=\frac{2^{\beta-1}\beta\,\Gamma((n+\beta)/2)}{\pi^{n/2}\Gamma(1-\beta/2)},
		\]
		where $\Gamma$ denotes the Gamma function. The space $\mathcal{F}_\beta$ is exactly the fractional Sobolev space $W^{\beta/2,2}(\mathbb{R}^n)$.
		
		For every $0<\beta<2$, $(\mathcal{E}_\beta,\mathcal{F}_\beta)$ is a regular non-local Dirichlet form, and its associated Laplacian is the fractional Laplacian $-(-\Delta)^{\beta/2}$.
		
		\item[(iii)] \textit{Sobolev spaces on Riemannian manifolds.} Let $(M,g)$ be a geodesically complete Riemannian manifold with volume measure $\mu$, and denote by $(\cdot,\cdot)_g$ the Riemannian inner product induced by $g$. The quadratic form
		\[
		\mathcal{E}(u,v):=\int_M\left(\nabla u,\nabla v\right)_g\,d\mu,\qquad u,v\in C_c^\infty(M),
		\]
		is closable in $L^2(M,\mu)$, and its closure is a regular local Dirichlet form $(\mathcal{E},W^{1,2}(M))$. Its associated Laplacian is the Laplace--Beltrami operator $\Delta_g$ (see, e.g., \cite[Chapter 4 \& Theorem 11.5]{Grigoryan.2009.482}).
		
		\item[(iv)] \textit{Dirichlet forms on the Sierpi\'nski gasket and the Sierpi\'nski carpet.} Let $K$ be the Sierpi\'nski gasket in $\mathbb{R}^n$, and let $\mu$ be the Hausdorff measure on $K$. Kigami \cite{Kigami.1993.TAMS721} constructed a regular self-similar Dirichlet form on $L^2(K,\mu)$ by building a compatible sequence of graph energies on approximating graphs; this construction was later extended to general post-critically finite (p.c.f.) self-similar fractals \cite{Kigami.2001.226}. For non-p.c.f. fractals such as the Sierpi\'nski carpet, the construction of a regular self-similar Dirichlet form is more delicate; see, e.g., \cite{KusuokaYin.1992.PTRF169}.
	\end{enumerate}
\end{example}

\subsection{The extended Dirichlet space $\mathcal{F}_e$}
\label{sub:Fe}
While $(\mathcal{F},\|\cdot\|_{\mathcal{E}_1})$ is a Hilbert space by virtue of the closedness of the form, $\mathcal{F}$ is not necessarily complete with respect to the energy norm $\sqrt{\mathcal{E}}$ alone. This causes a technical obstruction in the limiting procedure of Sections~\ref{Sec:Epi} and~\ref{Sec:Para}, where $\mathcal{E}$-Cauchy sequences arise.

To accommodate $\mathcal{E}$-Cauchy sequences that may not converge within $\mathcal{F}$, we enlarge $\mathcal{F}$ to $\mathcal{F}_e$ as follows.

\begin{definition}[Extended Dirichlet space]
	\label{def:Fe}
	Let $\mathcal{F}_e$ be the set of all measurable functions $u$ for which there exists an $\mathcal{E}$-Cauchy sequence $\{u_k\}_{k\ge 1}\subset \mathcal{F}$ such that $u_k \to u$ $\mu$-a.e. We call such a sequence an \emph{approximating sequence} of $u$. For $u\in\mathcal{F}_e$, set
	\[
	\mathcal{E}(u,u):=\lim_{k\to\infty}\mathcal{E}(u_k,u_k),
	\]
	which is well-defined independently of the choice of the approximating sequence $\{u_k\}_{k\ge1}$ (see \cite[Theorem 1.5.2(i)]{FukushimaOshimaTakeda.2011.489}). The pair $(\mathcal{F}_e,\mathcal{E})$ is called the \emph{extended Dirichlet space} of $(\mathcal{E},\mathcal{F})$. In the sequel, we shall also refer to $\mathcal{F}_e$ itself as the extended Dirichlet space when no confusion arises.
\end{definition}

The definition directly implies the following basic property.

\begin{proposition}
	\label{Fe}
	Let $u\in\mathcal{F}_e$ and let $\{u_n\}_{n\ge 1}$ be an approximating sequence of $u$. Then
	\[
	\lim_{n\to\infty}\mathcal{E}(u_n-u)=0.  \label{fe}
	\]
\end{proposition}

\begin{proof}
	For each fixed $n\ge 1$, the sequence $\{u_n-u_m\}_{m\ge 1}$ is an approximating sequence of $u_n-u$. Hence, by the definition of $(\mathcal{E},\mathcal{F}_e)$,
	\[
	\mathcal{E}(u_n-u)=\lim_{m\to\infty}\mathcal{E}(u_n-u_m).
	\]
	Since $\{u_n\}_{n\ge1}$ is an $\mathcal{E}$-Cauchy sequence, for every $\varepsilon>0$ there exists $N$ such that
	\[
	\mathcal{E}(u_n-u_m)<\frac{\varepsilon}{2}
	\]
	for all $n,m\ge N$. Therefore,
	\[
	\mathcal{E}(u_n-u)=\lim_{m\to\infty}\mathcal{E}(u_n-u_m)\le \frac{\varepsilon}{2}<\varepsilon
	\]
	for every $n\ge N$, which proves the claim.
\end{proof}

The following lemma extends the conclusion of Proposition~\ref{Fe} to $\mathcal{E}$-Cauchy sequences in the larger space $\mathcal{F}_e$.

\begin{lemma}
	\label{lem:Fe_closed}
	Let $\{u_n\}_{n\ge 1}\subset \mathcal{F}_e$ be an $\mathcal{E}$-Cauchy sequence such that $u_n\to u$ $\mu$-a.e. for some measurable function $u$. Then $u\in\mathcal{F}_e$ and
	\[
	\lim_{n\to\infty}\mathcal{E}(u_n-u)=0.
	\]
\end{lemma}

Since the proof of Lemma~\ref{lem:Fe_closed} is lengthy and technical, we relegate it to Appendix~\ref{App:A} to avoid interrupting the main flow.

In general, the space $(\mathcal{F}_e,\sqrt{\mathcal{E}})$ is not complete; indeed, there exists a Dirichlet form for which the extended space is not complete in the energy seminorm (see \cite[Corollary 6.3]{Schmidt2022}). Thus $\mathcal{F}_e$ is not the ultimate solution to the completeness problem.

Despite the incompleteness of $(\mathcal{F}_e,\sqrt{\mathcal{E}})$, the intersection $\mathcal{F}_e\cap L^2$ recovers the original domain $\mathcal{F}$ (see \cite[Theorem 1.5.2]{FukushimaOshimaTakeda.2011.489}):
\begin{equation}
	\mathcal{F}=\mathcal{F}_e\cap L^2(M,\mu).  \label{eq:FeL2}
\end{equation}

The identity \eqref{eq:FeL2} shows that, when $q=2$, intersecting $\mathcal{F}_e$ with $L^q$ gives back the original form domain $\mathcal{F}$, which is complete under the $\mathcal{E}_1$-norm. For $q\neq 2$, however, the space $L^q\cap\mathcal{F}_e$ is not covered by the standard theory and, to the best of our knowledge, has not been systematically studied in this generality. In the next section, we shall systematically study these intersection spaces and prove that they are reflexive Banach spaces under suitable norms—a fact that will be essential for the Minty--Browder argument in Section~\ref{Sec:Epi}.

\section{Functional framework: the spaces $V^q \coloneqq L^q \cap \mathcal{F}_e$}
\label{Sec:Vq}

Motivated by the discussion at the end of Section~\ref{Sec:DF}, for every \(1<q<\infty\), we introduce the intersection space
\[
V^q := L^q \cap \mathcal{F}_e,
\]
and define the functional
\begin{equation}
	\|u\|_{V^q} := \left( \|u\|_{L^q}^2 + \mathcal{E}(u) \right)^{1/2}.
	\label{eq:VqNorm}
\end{equation}

\subsection{Basic properties of $V^q$}

In this subsection we prove that \((V^q,\|\cdot\|_{V^q})\) is a reflexive Banach space for every \(1<q<\infty\). To this end, we first prove completeness and then establish uniform convexity, which implies reflexivity by the Milman--Pettis theorem.

\begin{proposition}
	\label{prop:Vq_norm}
	For every \(1<q<\infty\), \(\|\cdot\|_{V^q}\) is a norm on \(V^q\).
\end{proposition}

\begin{proof}
	Homogeneity and nonnegativity are immediate from the definition. Moreover, since
	\[
	\|v\|_{L^q} \le \|v\|_{V^q} \qquad \forall v\in V^q,
	\]
	we have \(\|v\|_{V^q}=0\) if and only if \(v=0\) \(\mu\)-a.e.
	
	It remains to verify the triangle inequality. Let \(v,w\in V^q\). Then
	\begin{align}
		\|v+w\|_{V^q}^2 &= \|v+w\|_{L^q}^2 + \mathcal{E}(v+w) \notag \\
		&\le \bigl(\|v\|_{L^q}+\|w\|_{L^q}\bigr)^2 + \bigl(\mathcal{E}(v)+\mathcal{E}(w)+2\mathcal{E}(v,w)\bigr) \notag \\
		&= \|v\|_{V^q}^2 + \|w\|_{V^q}^2 + 2\bigl(\|v\|_{L^q}\|w\|_{L^q} + \mathcal{E}(v,w)\bigr). \label{eq:tri-1}
	\end{align}
	On the other hand,
	\begin{align}
		\bigl(\|v\|_{V^q}+\|w\|_{V^q}\bigr)^2 &= \|v\|_{V^q}^2 + \|w\|_{V^q}^2 + 2\|v\|_{V^q}\|w\|_{V^q} \notag \\
		&= \|v\|_{V^q}^2 + \|w\|_{V^q}^2 + 2\sqrt{\|v\|_{L^q}^2+\mathcal{E}(v)} \sqrt{\|w\|_{L^q}^2+\mathcal{E}(w)}. \label{eq:tri-2}
	\end{align}
	Subtracting \eqref{eq:tri-1} from \eqref{eq:tri-2} yields
	\begin{equation}
		\begin{aligned}
			&\bigl(\|v\|_{V^q}+\|w\|_{V^q}\bigr)^2 - \|v+w\|_{V^q}^2 \\
			&= 2\Big( \sqrt{\|v\|_{L^q}^2+\mathcal{E}(v)} \sqrt{\|w\|_{L^q}^2+\mathcal{E}(w)} - \|v\|_{L^q}\|w\|_{L^q} - \mathcal{E}(v,w) \Big).
		\end{aligned} \label{eq:tri-diff}
	\end{equation}
	By the Cauchy--Schwarz inequality in \(\mathbb{R}^2\) applied to the pairs \((\|v\|_{L^q}, \sqrt{\mathcal{E}(v)})\) and \((\|w\|_{L^q}, \sqrt{\mathcal{E}(w)})\), we have
	\[
	\sqrt{\|v\|_{L^q}^2+\mathcal{E}(v)} \sqrt{\|w\|_{L^q}^2+\mathcal{E}(w)} \ge \|v\|_{L^q}\|w\|_{L^q} + \sqrt{\mathcal{E}(v)\mathcal{E}(w)}.
	\]
	Since \(|\mathcal{E}(v,w)| \le \sqrt{\mathcal{E}(v)\mathcal{E}(w)}\) by the Cauchy--Schwarz inequality for \(\mathcal{E}\), it follows that
	\[
	\sqrt{\mathcal{E}(v)\mathcal{E}(w)} \ge \mathcal{E}(v,w).
	\]
	Combining these estimates gives
	\[
	\sqrt{\|v\|_{L^q}^2+\mathcal{E}(v)} \sqrt{\|w\|_{L^q}^2+\mathcal{E}(w)} \ge \|v\|_{L^q}\|w\|_{L^q} + \mathcal{E}(v,w).
	\]
	Inserting this into \eqref{eq:tri-diff} yields
	\[
	\bigl(\|v\|_{V^q}+\|w\|_{V^q}\bigr)^2 - \|v+w\|_{V^q}^2 \ge 0,
	\]
	which is precisely the triangle inequality. This completes the proof.
\end{proof}

Having established that \(\|\cdot\|_{V^q}\) is a norm, we now address completeness.

\begin{proposition}
	\label{prop:Vq_Banach}
	For every \(1<q<\infty\), \((V^q,\|\cdot\|_{V^q})\) is a Banach space.
\end{proposition}

\begin{proof}
	Let \(\{v_n\}_{n\ge1}\) be a Cauchy sequence in \(V^q\). Then
	\[
	\lim_{n,m\to\infty} \|v_n-v_m\|_{V^q}
	= \lim_{n,m\to\infty} \sqrt{\|v_n-v_m\|_{L^q}^2 + \mathcal{E}(v_n-v_m)}
	= 0.
	\]
	In particular, \(\{v_n\}_{n\ge1}\) is Cauchy in \(L^q\) and \(\mathcal{E}\)-Cauchy in \(\mathcal{F}_e\). Since \(L^q\) is complete, there exists \(v\in L^q\) such that
	\begin{equation}
		\lim_{n\to\infty} \|v_n-v\|_{L^q}=0. \label{eq:Lq-limit}
	\end{equation}
	Hence \(v_n\to v\) in measure; passing to a subsequence if necessary, we may assume \(v_{k_n}\to v\) \(\mu\)-a.e. as \(n\to\infty\). Since \(\{v_{k_n}\}_{n\ge1}\) is also \(\mathcal{E}\)-Cauchy, Lemma~\ref{lem:Fe_closed} gives \(v\in\mathcal{F}_e\) and
	\begin{equation}
		\lim_{n\to\infty} \mathcal{E}(v_{k_n}-v)=0. \label{eq:Ek-limit}
	\end{equation}
	
	It remains to show that the full sequence \(\{v_n\}_{n\ge1}\) converges with respect to the energy norm. To this end, fix \(\varepsilon>0\). Since \(\{v_n\}_{n\ge1}\) is \(\mathcal{E}\)-Cauchy, there exists \(N_\varepsilon\) such that
	\[
	\sqrt{\mathcal{E}(v_l-v_{k_n})} < \varepsilon \quad\text{whenever } l,n \ge N_\varepsilon.
	\]
	Letting \(n\to\infty\) in the above inequality, we obtain, by the triangle inequality for \(\mathcal{E}\) and \eqref{eq:Ek-limit},
	\[
	\sqrt{\mathcal{E}(v_l-v)} \le \lim_{n\to\infty} \left( \sqrt{\mathcal{E}(v_l-v_{k_n})} + \sqrt{\mathcal{E}(v_{k_n}-v)} \right) \le \varepsilon\quad\text{ for any }l\ge N_\varepsilon.
	\]
	Hence
	\[
	\limsup_{l\to\infty} \sqrt{\mathcal{E}(v_l-v)} \le \varepsilon.
	\]
	Since \(\varepsilon>0\) was arbitrary, we conclude
	\[
	\lim_{l\to\infty} \mathcal{E}(v_l-v)=0.
	\]
	Together with \eqref{eq:Lq-limit}, this gives
	\[
	\lim_{n\to\infty} \|v_n-v\|_{V^q}
	= \lim_{n\to\infty} \sqrt{\|v_n-v\|_{L^q}^2 + \mathcal{E}(v_n-v)}
	= 0.
	\]
	Thus \((V^q,\|\cdot\|_{V^q})\) is complete.
\end{proof}

The Banach space structure alone, however, is not sufficient for the application of the Minty--Browder theorem; we also require reflexivity. A convenient route to reflexivity is to establish uniform convexity, which we do in the following proposition.

\begin{proposition}
	\label{prop:Vq_uniformly_convex}
	For every $1<q<\infty$, $(V^q,\|\cdot\|_{V^q})$ is uniformly convex.
\end{proposition}

\begin{proof}
	Fix $\varepsilon>0$ and take $v,w\in V^q$ such that
	\begin{equation}
		\|v\|_{V^q}=\|w\|_{V^q}=1,\qquad \|v-w\|_{V^q}>\varepsilon. \label{eq:uc-assume}
	\end{equation}
	From the definition of $\|\cdot\|_{V^q}$ and the triangle inequality for $\|\cdot\|_{L^q}$, we have
	\begin{align}
		\left\|\frac{v+w}{2}\right\|_{V^q}^2
		&= \frac{1}{4}\|v+w\|_{V^q}^2
		= \frac{1}{4}\left(\|v+w\|_{L^q}^2 + \mathcal{E}(v+w)\right) \notag \\
		&\le \frac{1}{4}\left(\bigl(\|v\|_{L^q}+\|w\|_{L^q}\bigr)^2 + \mathcal{E}(v+w)\right). \label{eq:uc-e}
	\end{align}
	Using the identities
	\[
	\mathcal{E}(v+w)=2\mathcal{E}(v)+2\mathcal{E}(w)-\mathcal{E}(v-w),
	\qquad
	\left(\|v\|_{L^q}+\|w\|_{L^q}\right)^2=2\|v\|_{L^q}^2+2\|w\|_{L^q}^2-\left(\|v\|_{L^q}-\|w\|_{L^q}\right)^2,
	\]
	we obtain from \eqref{eq:uc-e} that
	\begin{align}
		\left\|\frac{v+w}{2}\right\|_{V^q}^2
		&\le \frac{1}{4}\left(2\|v\|_{L^q}^2+2\|w\|_{L^q}^2-\bigl(\|v\|_{L^q}-\|w\|_{L^q}\bigr)^2+2\mathcal{E}(v)+2\mathcal{E}(w)-\mathcal{E}(v-w)\right) \notag \\
		&= \frac{1}{4}\left(2\|v\|_{V^q}^2+2\|w\|_{V^q}^2-\bigl(\|v\|_{L^q}-\|w\|_{L^q}\bigr)^2-\mathcal{E}(v-w)\right) \notag \\
		&= \frac{1}{4}\left(4-\bigl(\|v\|_{L^q}-\|w\|_{L^q}\bigr)^2-\mathcal{E}(v-w)\right). \label{eq:uc-estimate}
	\end{align}
	
	We now distinguish two cases.
	
	\medskip
	
	\noindent\textbf{Case 1:} $\mathcal{E}(v-w)>\varepsilon^2/4$.
	
	Then from \eqref{eq:uc-estimate} we get
	\begin{equation} \label{eq:ucd01}
		\left\|\frac{v+w}{2}\right\|_{V^q}^2 \le \frac{1}{4}\left(4-\frac{\varepsilon^2}{4}\right)=1-\frac{\varepsilon^2}{16}.
	\end{equation}
	
	\medskip
	
	\noindent\textbf{Case 2:} $\mathcal{E}(v-w)\le \varepsilon^2/4$.
	
	Since $\|v-w\|_{V^q}^2=\|v-w\|_{L^q}^2+\mathcal{E}(v-w)>\varepsilon^2$, we have
	\begin{equation}
		\|v-w\|_{L^q}>\frac{\varepsilon}{2}. \label{eq:uc-lq-bound}
	\end{equation}
	
	If $\bigl|\|v\|_{L^q}-\|w\|_{L^q}\bigr|>\varepsilon/4$, then \eqref{eq:uc-estimate} yields
	\begin{equation} \label{eq:ucd02}
		\left\|\frac{v+w}{2}\right\|_{V^q}^2 \le \frac{1}{4}\left(4-\left(\frac{\varepsilon}{4}\right)^2\right)=1-\frac{\varepsilon^2}{64}.
	\end{equation}
	Thus we may assume
	\begin{equation*}
		\bigl|\|v\|_{L^q}-\|w\|_{L^q}\bigr|\le \frac{\varepsilon}{4}.
	\end{equation*}
	
	Set $a:=\|v\|_{L^q}$ and $b:=\|w\|_{L^q}$. Then $a,b\in[0,1]$. By symmetry we may assume $a\ge b$. We claim $b>0$; indeed, if $b=0$, then $w=0$ $\mu$-a.e., so $a=\|v-w\|_{L^q}>\varepsilon/2$, but $a-b=a\le \varepsilon/4$, a contradiction.
	
	Define
	\[
	\tilde v:=\frac{v}{a},\qquad \tilde w:=\frac{w}{b}.
	\]
	Then $\|\tilde v\|_{L^q}=\|\tilde w\|_{L^q}=1$.
	
	From \eqref{eq:uc-lq-bound} and the triangle inequality,
	\[
	\frac{\varepsilon}{2}<\|v-w\|_{L^q}=\|a\tilde v-b\tilde w\|_{L^q}
	\le |a-b|\|\tilde v\|_{L^q}+b\|\tilde v-\tilde w\|_{L^q}
	\le \frac{\varepsilon}{4}+\|\tilde v-\tilde w\|_{L^q},
	\]
	hence
	\[
	\|\tilde v-\tilde w\|_{L^q}>\frac{\varepsilon}{4}.
	\]
	By the uniform convexity of $L^q(M,\mu)$, there exists $\delta_q\in(0,1)$, depending only on $q$ and $\varepsilon$, such that
	\[
	\left\|\frac{\tilde v+\tilde w}{2}\right\|_{L^q}\le 1-\delta_q.
	\]
	
	Now,
	\begin{align}
		\left\|\frac{v+w}{2}\right\|_{L^q}
		&= \left\|\frac{a\tilde v+b\tilde w}{2}\right\|_{L^q} \le \frac{a-b}{2}\|\tilde v\|_{L^q}+b\left\|\frac{\tilde v+\tilde w}{2}\right\|_{L^q} \notag \\
		&\le \frac{a-b}{2}+b(1-\delta_q)=\frac{a+b}{2}-b\delta_q. \label{eq:uc-norm-est}
	\end{align}
	Using $a^2+\mathcal{E}(v)=1$ and $b^2+\mathcal{E}(w)=1$, we get
	\begin{align}
		\left\|\frac{v+w}{2}\right\|_{V^q}^2
		&= \left\|\frac{v+w}{2}\right\|_{L^q}^2+\mathcal{E}\left(\frac{v+w}{2}\right) \notag \\
		&\le \left(\frac{a+b}{2}-b\delta_q\right)^2+\frac{\mathcal{E}(v)+\mathcal{E}(w)}{2} \notag \\
		&= \left(\frac{a+b}{2}-b\delta_q\right)^2+\frac{(1-a^2)+(1-b^2)}{2} \notag \\
		&\le 1-ab\delta_q-b^2\delta_q(1-\delta_q) \notag \\
		&\le 1-ab\delta_q \le 1-b^2\delta_q, \label{eq:uc-final}
	\end{align}
	where the last two inequalities follow from $0<\delta_q<1$ and $0<b\le a$, respectively. Thus
	\[
	\left\|\frac{v+w}{2}\right\|_{V^q}\le \sqrt{1-b^2\delta_q}.
	\]
	
	Finally, from $a-b\le \varepsilon/4$ and $a+b\ge \|v-w\|_{L^q}>\varepsilon/2$, we obtain $b>\varepsilon/8$. Hence
	\begin{equation} \label{ucd03}
		\left\|\frac{v+w}{2}\right\|_{V^q}<\sqrt{1-\frac{\varepsilon^2}{64}\delta_q}.
	\end{equation}
	
	Combining \eqref{eq:ucd01}, \eqref{eq:ucd02} and \eqref{ucd03}, for every $\varepsilon>0$, there exists
	\[
	\delta:=1-\sqrt{1-\frac{\varepsilon^2}{64}\delta_q}\in(0,1),
	\]
	depending only on $\varepsilon$, such that whenever $\|v\|_{V^q}=\|w\|_{V^q}=1$ and $\|v-w\|_{V^q}>\varepsilon$, we have
	\[
	\left\|\frac{v+w}{2}\right\|_{V^q}\le 1-\delta.
	\]
	By definition, $(V^q,\|\cdot\|_{V^q})$ is uniformly convex.
\end{proof}

\begin{corollary}
	\label{cor:Vq_reflexive}
	For every $1<q<\infty$, $V^q$ is a reflexive Banach space.
\end{corollary}

\begin{proof}
	This follows immediately from Propositions \ref{prop:Vq_Banach} and \ref{prop:Vq_uniformly_convex}, together with the Milman--Pettis theorem.
\end{proof}

\subsection{Commutation lemmas among weak limits, energy and Bochner integrals}

To implement the limiting procedure in Section~\ref{Sec:Para}, we shall need to interchange the order of several operations: weak limits, the Dirichlet energy $\mathcal{E}$, and time integrals. The three lemmas collected below provide the necessary commutation rules.

The lemmas are stated in a general setting that includes our situation as a special case: they hold for any Banach space $X$ and any symmetric non-negative bilinear form $\mathcal{B}:X\times X\to\mathbb{R}$ satisfying
\begin{equation}
	|\mathcal{B}(x,y)|\le C\|x\|_X\|y\|_X\qquad \forall x,y\in X, \label{eq:bounded-bilinear}
\end{equation}
for some constant $C>0$ independent of $x,y$. In particular, for each fixed $y\in X$, the linear functional
\[
\mathcal{B}(\cdot,y): x\mapsto \mathcal{B}(x,y)
\]
is bounded on $X$, with \(\|\mathcal{B}(\cdot, y)\|_{X^*}\le C\|y\|_X\) and
\begin{equation}\label{eq:pullback}
	\mathcal{B}(x,y)=\langle\mathcal{B}(\cdot,y),x\rangle_X.
\end{equation}

The first lemma deals with the interchange of limits and the energy functional.

\begin{lemma}[Lower semicontinuity of the energy]
	\label{lem:lsc}
	Assume \eqref{eq:bounded-bilinear}. If $\{x_k\}\subset X$ converges weakly to $x$ in $X$ as $k\to\infty$, then
	\begin{equation}
		\mathcal{B}(x,x)\le \liminf_{k\to\infty}\mathcal{B}(x_k,x_k).
		\label{eq:lsc}
	\end{equation}
\end{lemma}

\begin{proof}
	If $\mathcal{B}(x,x)=0$, then \eqref{eq:lsc} is trivial. Suppose $\mathcal{B}(x,x)>0$. By \eqref{eq:bounded-bilinear}, the map $\mathcal{B}(\cdot,x)$ is a bounded linear functional on $X$. Hence, from $x_k\rightharpoonup x$ in $X$, we have
	\begin{equation}
		\lim_{k\to\infty}\mathcal{B}(x_k,x)=\mathcal{B}(x,x). \label{eq:lsc-weak}
	\end{equation}
	On the other hand, the Cauchy--Schwarz inequality for $\mathcal{B}$ gives
	\[
	\mathcal{B}(x_k,x)^2\le \mathcal{B}(x_k,x_k)\mathcal{B}(x,x).
	\]
	Taking the lower limit as $k\to\infty$ and using \eqref{eq:lsc-weak}, we obtain
	\[
	\mathcal{B}(x,x)^2\le \mathcal{B}(x,x)\cdot \liminf_{k\to\infty}\mathcal{B}(x_k,x_k).
	\]
	Since $\mathcal{B}(x,x)>0$, dividing by $\mathcal{B}(x,x)$ yields \eqref{eq:lsc}. This completes the proof.
\end{proof}

The next lemma ensures that weak convergence in the Bochner space is preserved under time integration.

\begin{lemma}[Weak continuity of the Bochner integral]
	\label{lem:bochner-weak}
	Suppose that $\{\mathbf{x}_k\}_{k\ge1}\cup\{\mathbf{x}\}\subset L^2(0,T;X)$ and $\mathbf{x}_k\rightharpoonup \mathbf{x}$ in $L^2(0,T;X)$ as $k\to\infty$. Then, for every measurable interval $I\subset[0,T]$,
	\begin{equation}
		\int_I\mathbf{x}_k(t)\,dt\rightharpoonup \int_I\mathbf{x}(t)\,dt\quad\text{in }X\text{ as }k\to\infty. \label{eq:bochner-weak}
	\end{equation}
	
	Moreover, assume that \eqref{eq:bounded-bilinear} holds. Then for every measurable interval $I\subset[0,T]$,
	\begin{equation} \label{eq:Eliminf}
		\mathcal{B}\left(\int_I\mathbf{x}(t)\,dt,\int_I\mathbf{x}(t)\,dt\right)
		\le \liminf_{k\to\infty}\mathcal{B}\left(\int_I\mathbf{x}_k(t)\,dt,\int_I\mathbf{x}_k(t)\,dt\right).
	\end{equation}
\end{lemma}

\begin{proof}
	Fix $h\in X^*$. We prove that
	\begin{equation}
		\lim_{k\to\infty}\left\langle h,\int_I\mathbf{x}_k(t)\,dt\right\rangle_X
		= \left\langle h,\int_I\mathbf{x}(t)\,dt\right\rangle_X. \label{eq:bochner-weak-h}
	\end{equation}
	By the Bochner's theorem (see, e.g., \cite[Theorem 8, Appendix E.5]{Evans.2010.749} or \cite[Theorem 23.9(a)]{Zeidler1990IIA}), for any $\mathbf{w}\in L^2(0,T;X)$,
	\begin{equation}
		\left\langle h,\int_I\mathbf{w}(t)\,dt\right\rangle_X=\int_I\langle h,\mathbf{w}(t)\rangle_X\,dt. \label{eq:bochner-h}
	\end{equation}
	Define \(\mathbf{h}^*\) by \(\mathbf{h}^*(t):=h\mathbf{1}_I(t)\). Then \(\mathbf{h}^*\in L^2(0,T;X^*)\) and
	\begin{equation}
		\int_I\langle h,\mathbf{w}(t)\rangle_X\,dt
		= \int_0^T\langle \mathbf{h}^*(t),\mathbf{w}(t)\rangle_X\,dt
		= \langle \mathbf{h}^*,\mathbf{w}\rangle_{L^2(0,T;X)}
		\quad\text{for all }\mathbf{w}\in L^2(0,T;X). \label{eq:bochner-hstar}
	\end{equation}
	Taking $\mathbf{w}=\mathbf{x}_k$ in \eqref{eq:bochner-h} and \eqref{eq:bochner-hstar}, we obtain
	\[
	\left\langle h,\int_I\mathbf{x}_k(t)\,dt\right\rangle_X
	= \langle \mathbf{h}^*,\mathbf{x}_k\rangle_{L^2(0,T;X)}.
	\]
	Letting $k\to\infty$ and using $\mathbf{x}_k\rightharpoonup \mathbf{x}$ in $L^2(0,T;X)$, we get
	\[
	\lim_{k\to\infty}\left\langle h,\int_I\mathbf{x}_k(t)\,dt\right\rangle_X
	= \langle \mathbf{h}^*,\mathbf{x}\rangle_{L^2(0,T;X)}
	= \left\langle h,\int_I\mathbf{x}(t)\,dt\right\rangle_X,
	\]
	which proves \eqref{eq:bochner-weak-h}. Since $h\in X^*$ was arbitrary, \eqref{eq:bochner-weak} follows.
	
	Finally, applying \eqref{eq:bochner-weak} to Lemma~\ref{lem:lsc}, we obtain \eqref{eq:Eliminf}.
\end{proof}

The final lemma provides a Fubini-type rule for interchanging $\mathcal{B}$ with a double time integral.

\begin{lemma}[Fubini identity for $\mathcal{B}$]
	\label{lem:fubini}
	Assume \eqref{eq:bounded-bilinear}. For any $\mathbf{x}\in L^2(0,T;X)$ and any measurable intervals $I,J\subset[0,T]$, there holds
	\begin{equation}
		\mathcal{B}\left(\int_I\mathbf{x}(t)\,dt,\int_J\mathbf{x}(s)\,ds\right)
		= \int_I\mathcal{B}\left(\mathbf{x}(t),\int_J\mathbf{x}(s)\,ds\right)dt
		= \int_I\int_J\mathcal{B}(\mathbf{x}(t),\mathbf{x}(s))\,ds\,dt.
		\label{eq:fubini}
	\end{equation}
\end{lemma}

\begin{proof}
By \eqref{eq:pullback} and the Bochner theorem applied to the bounded linear functional \(\mathcal{B}\left(\cdot,\int_J\mathbf{x}(s)\,ds\right)\), we have
\begin{align}
	\mathcal{B}\left(\int_I\mathbf{x}(t)\,dt,\int_J\mathbf{x}(s)\,ds\right)
	&= \left\langle \mathcal{B}\left(\cdot,\int_J\mathbf{x}(s)\,ds\right),\int_I\mathbf{x}(t)\,dt\right\rangle_X \notag \\
	&= \int_I\mathcal{B}\left(\mathbf{x}(t),\int_J\mathbf{x}(s)\,ds\right)dt. \label{eq:fubini-first}
\end{align}
Similarly, for each fixed \(t\in I\), it follows from the symmetry of \(\mathcal{B}\) and the Bochner theorem that
\begin{equation}
	\mathcal{B}\left(\mathbf{x}(t),\int_J\mathbf{x}(s)\,ds\right)
	= \mathcal{B}\left(\int_J\mathbf{x}(s)\,ds,\mathbf{x}(t)\right)
	= \int_J\mathcal{B}(\mathbf{x}(s),\mathbf{x}(t))\,ds
	= \int_J\mathcal{B}(\mathbf{x}(t),\mathbf{x}(s))\,ds. \label{eq:fubini-second}
\end{equation}
Integrating both sides of \eqref{eq:fubini-second} with respect to \(t\) on \(I\), we finally obtain \eqref{eq:fubini}. This completes the proof.
\end{proof}

The three lemmas above provide the commutation rules between weak limits, $\mathcal{B}$ (and hence $\mathcal{E}$) and time integrations, which will be essential in the Minty trick argument of Section~\ref{Sec:Para}. Their proofs rely only on \eqref{eq:bounded-bilinear} and are independent of the particular structure of $V^q$.

For the rest of the paper, whenever $m>0$ is fixed, we write
\begin{equation} \label{eq:V-special}
	V:=V^{1+1/m}\qquad\text{and}\qquad \|\cdot\|_V:=\|\cdot\|_{V^{1+1/m}}.
\end{equation}
Since $1+1/m\in(1,\infty)$ for $m\in(0,\infty)$, it follows from Corollary~\ref{cor:Vq_reflexive} that $V$ is a reflexive Banach space. Both the reflexivity and the completeness of $V$ will be used in the subsequent sections when applying the Minty--Browder theorem.

\section{Well-posedness of the time-discrete elliptic problem}
\label{Sec:Epi}

In this section, we study the elliptic problem obtained by time-discretization of the evolution equation \(\partial_t u = \mathcal{L}(\Psi(u))\). The results established here will be used in Section~\ref{Sec:Para} to construct approximate solutions via the Rothe method and to derive the uniform estimates needed for the passage to the limit.

Let \(m>0\) and \(T>0\) be fixed, and let \(n\in\mathbb{N}\). A backward Euler discretization of the equation on \([0,T)\) leads to the following recursive problem: given \(u_{k-1}\in L^{m+1}(M,\mu)\), find \(u_k\) such that
\[
\frac{u_k-u_{k-1}}{T/n} = \mathcal{L}(\Psi(u_k)).
\]
Equivalently,
\begin{equation}
	u_k - \frac{T}{n}\mathcal{L}(\Psi(u_k)) = u_{k-1}. \label{eq:discrete-u}
\end{equation}
This is a nonlinear elliptic equation of the form
\begin{equation}
	u - \tau \mathcal{L}(\Psi(u)) = g, \label{eq:elliptic-u}
\end{equation}
where \(u:=u_k\) is the unknown, \(g:=u_{k-1}\) is the data from the previous time step, and \(\tau:=T/n\) is the time step. Setting
\[
v:=\Psi(u)=|u|^{m-1}u,
\]
we have \(u=\Psi^{-1}(v)=|v|^{1/m}\operatorname{sgn}(v)\), and \eqref{eq:elliptic-u} becomes
\begin{equation}
	\Psi^{-1}(v) - \tau \mathcal{L}(v) = g. \label{eq:elliptic-v}
\end{equation}

We shall study \eqref{eq:elliptic-v} in the reflexive Banach space \(V:=V^{1+1/m}\) defined in \eqref{eq:V-special}.

\subsection{Weak formulation and the Minty--Browder theorem}

We begin with the notion of weak solution to \eqref{eq:elliptic-v}.

\begin{definition}
	\label{def:elliptic_weak}
	Let \(\tau>0\) and \(m>0\) be given, and let \(g\in V^*\). A function \(v\in V\) is called a \emph{weak solution} of \eqref{eq:elliptic-v} if
	\begin{equation}
		\int_M \Psi^{-1}(v)\,\phi\,d\mu + \tau\,\mathcal{E}(v,\phi) = \langle g,\phi\rangle_V \label{eq:elliptic-weak}
	\end{equation}
	for every \(\phi\in V\).
\end{definition}

The Minty--Browder theorem will be our main tool for proving existence and uniqueness. We recall the following standard result (see \cite[Theorem 26.A]{Zeidler1990IIB}). We note that, in the version stated below, separability of \(X\) is not required.

\begin{theorem}[Minty--Browder theorem]
	\label{thm:minty-browder}
	Let \(X\) be a reflexive real Banach space, and let \(A:X\to X^*\) be an operator satisfying:
	
	\begin{enumerate}
		\item[(i)] \emph{Monotonicity:} for all \(v,w\in X\),
		\[
		\langle Av-Aw, v-w\rangle_X \ge 0;
		\]
		
		\item[(ii)] \emph{Hemicontinuity:} for all \(v,w,h\in X\), the map
		\[
		t \mapsto \langle A(v+tw), h\rangle_X
		\]
		is continuous on \([0,1]\);
		
		\item[(iii)] \emph{Coercivity:}
		\[
		\lim_{\|v\|_X\to\infty} \frac{\langle Av,v\rangle_X}{\|v\|_X} = \infty.\]
	\end{enumerate}
	
	Then, for every \(g\in X^*\), there exists \(v_0\in X\) such that \(Av_0=g\). If \(A\) is strictly monotone (i.e., \(\langle Av-Aw, v-w\rangle_X=0\) implies \(v=w\)), then \(v_0\) is unique.
\end{theorem}

We now apply the theorem to the operator associated with \eqref{eq:elliptic-v}.

\begin{lemma}
	\label{lem:elliptic_existence}
	Let \(\tau>0\) and \(m>0\) be given. Then, for every \(g\in V^*\), the equation \eqref{eq:elliptic-v} admits a unique weak solution \(v\in V\).
\end{lemma}

\begin{proof}
	Define the operator \(A:V\to V^*\) by
	\begin{equation}
		\langle Av,\phi\rangle_V := \int_M \Psi^{-1}(v)\,\phi\,d\mu + \tau\,\mathcal{E}(v,\phi) \label{eq:A-def}
	\end{equation}
	for \(v,\phi\in V\).
	
	The operator \(A\) is well-defined and bounded. Moreover, a weak solution of \eqref{eq:elliptic-v} is precisely a solution of \(Av=g\). We now verify that \(A\) satisfies the hypotheses of Theorem~\ref{thm:minty-browder}.
	
	\medskip
	
	\noindent\textbf{Strict monotonicity.} For \(v,w\in V\), we have
	\begin{align}
		\langle Av-Aw, v-w\rangle_V
		&= \int_M (\Psi^{-1}(v)-\Psi^{-1}(w))(v-w)\,d\mu + \tau\,\mathcal{E}(v-w,v-w) \notag \\
		&\ge \int_M (\Psi^{-1}(v)-\Psi^{-1}(w))(v-w)\,d\mu \ge 0, \label{eq:mono}
	\end{align}
	where the last inequality follows from the strict monotonicity of the function \(\Psi^{-1}(s)=|s|^{1/m}\operatorname{sgn}(s)\). Moreover, equality in \eqref{eq:mono} holds only if \(v=w\) \(\mu\)-a.e. Hence \(A\) is strictly monotone.
	
	\medskip
	
	\noindent\textbf{Hemicontinuity.} Fix \(v,w,\phi\in V\) and define
	\[
	F(t):=\langle A(v+tw),\phi\rangle_V
	= \int_M \Psi^{-1}(v+tw)\,\phi\,d\mu + \tau\,\mathcal{E}(v+tw,\phi), \qquad t\in[0,1].
	\]
	The second term is linear in \(t\), hence continuous. For the first term, note that for all \(t\in[-2,2]\),
	\[
	\left|\Psi^{-1}(v+tw)\,\phi\right| = |v+tw|^{1/m}|\phi| \le (|v|+2|w|)^{1/m}|\phi|,
	\]
	and the right-hand side is integrable by Hölder's inequality. Therefore, by the dominated convergence theorem and the continuity of \(\Psi^{-1}\), for any \(t_0\in[0,1]\),
	\[
	\lim_{t\to t_0}\int_M \Psi^{-1}(v+tw)\,\phi\,d\mu
	= \int_M \Psi^{-1}(v+t_0w)\,\phi\,d\mu.
	\]
	Thus \(F\) is continuous on \([0,1]\), proving hemicontinuity.
	
	\medskip
	
	\noindent\textbf{Coercivity.} For \(v\in V\), set \(a:=\|v\|_{L^{1+1/m}}\) and \(b:=\sqrt{\mathcal{E}(v)}\). Then \(\|v\|_V=\sqrt{a^2+b^2}\), and
	\[
	\frac{\langle Av,v\rangle_V}{\|v\|_V}
	= \frac{\|v\|_{L^{1+1/m}}^{1+1/m} + \tau\,\mathcal{E}(v)}{\sqrt{\|v\|_{L^{1+1/m}}^2 + \mathcal{E}(v)}}
	= \frac{a^{1+1/m} + \tau b^2}{\sqrt{a^2+b^2}}.
	\]
	We claim that
	\[
	\frac{a^{1+1/m} + \tau b^2}{\sqrt{a^2+b^2}}
	\ge
	\begin{cases}
		2^{-\frac{1+1/m}{2}}\,\left(\sqrt{a^2+b^2}\right)^{1/m}, & a\ge b,\\[4pt]
		\dfrac{\tau}{2}\,\sqrt{a^2+b^2}, & a<b.
	\end{cases}
	\]
	Indeed, if \(a\ge b\), then \(a^2\ge (a^2+b^2)/2\), so
	\[
	\frac{a^{1+1/m} + \tau b^2}{\sqrt{a^2+b^2}}
	\ge \frac{a^{1+1/m}}{\sqrt{a^2+b^2}}
	\ge \frac{((a^2+b^2)/2)^{\frac{1+1/m}{2}}}{\sqrt{a^2+b^2}}
	= 2^{-\frac{1+1/m}{2}}\left(\sqrt{a^2+b^2}\right)^{1/m}.
	\]
	If \(a<b\), then \(b^2\ge (a^2+b^2)/2\), so
	\[
	\frac{a^{1+1/m} + \tau b^2}{\sqrt{a^2+b^2}}
	\ge \frac{\tau b^2}{\sqrt{a^2+b^2}}
	\ge \frac{\tau}{2}\sqrt{a^2+b^2}.
	\]
	Consequently, for \(\|v\|_V\ge 1\),
	\[
	\frac{\langle Av,v\rangle_V}{\|v\|_V}
	\ge \min\left(2^{-\frac{1+1/m}{2}},\, \frac{\tau}{2}\right)
	\|v\|_V^{\min(1/m,\,1)},
	\]
	which tends to \(\infty\) as \(\|v\|_V\to\infty\). Thus \(A\) is coercive.
	
	All hypotheses of the Minty--Browder theorem are satisfied. Hence there exists a unique \(v\in V\) such that \(Av=g\). By Definition~\ref{def:elliptic_weak}, this \(v\) is the unique weak solution of \eqref{eq:elliptic-v}.
\end{proof}

\subsection{A priori estimates}

We now focus on the case where the datum \(g\in L^{m+1}(M,\mu)\). Since \(V\hookrightarrow L^{m+1}(M,\mu)\) continuously, the natural embedding \(L^{m+1}(M,\mu)\hookrightarrow V^*\) given by
\[
\langle g,\phi\rangle_V = \int_M g\,\phi\,d\mu \qquad \text{for all } \phi\in V
\]
is also continuous.

We first establish the comparison principle for the elliptic problem.

\begin{proposition}[Comparison principle for the elliptic problem]
	\label{prop:elliptic-comparison}
	Let \(g_1,g_2\in L^{m+1}(M,\mu)\) with \(g_1\le g_2\) \(\mu\)-a.e., and for \(i=1,2\), let \(v_i\in V\) be the corresponding weak solution of \eqref{eq:elliptic-v} with datum \(g=g_i\). Then \(v_1\le v_2\) \(\mu\)-a.e.
\end{proposition}

\begin{proof}
	Set \(w:=v_1-v_2\). Applying the weak formulation \eqref{eq:elliptic-weak} to \(v_1\) and \(v_2\) with the same test function \(\phi\in V\), and subtracting, yields
	\begin{equation}
		\int_M (\Psi^{-1}(v_1)-\Psi^{-1}(v_2))\,\phi\,d\mu + \tau\,\mathcal{E}(w,\phi)
		= \int_M (g_1-g_2)\,\phi\,d\mu \qquad \forall \phi\in V.
		\label{eq:comp-weak-general}
	\end{equation}
	It suffices to show that \(w_+=0\) \(\mu\)-a.e.
	
	Since \(w\in V=L^{1+1/m}\cap\mathcal{F}_e\) and \(t\mapsto t_+\) is a normal contraction, \cite[Corollary 1.6.3]{FukushimaOshimaTakeda.2011.489} yields \(w_+\in\mathcal{F}_e\). Moreover, \(0\le w_+\le |w|\) gives \(w_+\in L^{1+1/m}\), hence \(w_+\in V\).
	
	Take \(\phi=w_+\) in \eqref{eq:comp-weak-general}. Since \(g_1-g_2\le0\) and \(w_+\ge0\), we obtain
	\begin{equation}
		\int_M (\Psi^{-1}(v_1)-\Psi^{-1}(v_2))w_+\,d\mu + \tau\,\mathcal{E}(w,w_+) \le 0.
		\label{eq:comp-weak}
	\end{equation}
	
	We claim that the left-hand side of \eqref{eq:comp-weak} is non-negative. Indeed, since \(\Psi^{-1}\) is strictly increasing,
	\[
	(\Psi^{-1}(v_1)-\Psi^{-1}(v_2))w_+ \ge 0 \quad \mu\text{-a.e.},
	\]
	so the integral is non-negative. Moreover, using \(w=w_+-w_-\) and \(|w|=w_+ + w_-\), we have
	\[
	\mathcal{E}(w)=\mathcal{E}(w_+)+\mathcal{E}(w_-)-2\mathcal{E}(w_+,w_-),
	\]
	and
	\[
	\mathcal{E}(|w|)=\mathcal{E}(w_+)+\mathcal{E}(w_-)+2\mathcal{E}(w_+,w_-).
	\]
	The Markovian property gives \(\mathcal{E}(|w|)\le\mathcal{E}(w)\); comparing the two expressions yields
	\[
	\mathcal{E}(w_+,w_-)\le0.
	\]
	Therefore,
	\[
	\mathcal{E}(w,w_+)=\mathcal{E}(w_+-w_-,w_+)=\mathcal{E}(w_+)-\mathcal{E}(w_+,w_-)\ge0.
	\]
	Thus both terms on the left-hand side of \eqref{eq:comp-weak} are non-negative, while the right-hand side is non-positive. Hence equality must hold, and in particular
	\[
	\int_M (\Psi^{-1}(v_1)-\Psi^{-1}(v_2))w_+\,d\mu = 0.
	\]
	By the strict monotonicity of \(\Psi^{-1}\), this implies \(w_+=0\) \(\mu\)-a.e., i.e., \(v_1\le v_2\).
\end{proof}

\begin{corollary}[Nonnegativity preservation]
	\label{prop:nonnega}
	Let \(g\in L^{m+1}(M,\mu)\) with \(g\ge0\) \(\mu\)-a.e., and let \(v\in V\) be the weak solution of \eqref{eq:elliptic-v} with datum \(g\). Then \(v\ge0\) \(\mu\)-a.e.
\end{corollary}

\begin{proof}
	Let \(v_0\) be the weak solution of \eqref{eq:elliptic-v} with datum \(0\). Taking \(\phi=v_0\) in \eqref{eq:elliptic-weak}, we have
	\[
	\int_M \Psi^{-1}(v_0)v_0\,d\mu + \tau\,\mathcal{E}(v_0) = 0,
	\]
	which implies \(v_0=0\) \(\mu\)-a.e., since \(\int_M \Psi^{-1}(v_0)v_0\,d\mu = \|v_0\|_{L^{1+1/m}}^{1+1/m}\ge0\) and \(\mathcal{E}(v_0)\ge0\).
	
	Applying Proposition~\ref{prop:elliptic-comparison} with \(g_1=g\) and \(g_2=0\), we obtain \(v\ge v_0=0\) \(\mu\)-a.e.
\end{proof}

We now derive the energy estimate which will be essential for Section~\ref{Sec:Para}.

\begin{lemma}
	\label{lem:elliptic_estimates}
	Let \(g\in L^{m+1}(M,\mu)\) and let \(v\in V\) be the weak solution of \eqref{eq:elliptic-v} with datum $g$. Then
	\begin{equation}
		\|v\|_{L^{1+1/m}}^{1+1/m} + \tau(m+1)\mathcal{E}(v) \le \|g\|_{L^{m+1}}^{m+1}.
		\label{eq:elliptic-estimate}
	\end{equation}
\end{lemma}

\begin{proof}
	Taking \(\phi=v\) in \eqref{eq:elliptic-weak} gives
	\[
	\int_M \Psi^{-1}(v)v\,d\mu + \tau\,\mathcal{E}(v) = \int_M g\,v\,d\mu.
	\]
	By Young's inequality,
	\[
	\int_M g\,v\,d\mu \le \frac{1}{m+1}\|g\|_{L^{m+1}}^{m+1} + \frac{m}{m+1}\|v\|_{L^{1+1/m}}^{1+1/m}.
	\]
	Since \(\Psi^{-1}(v)v = |v|^{1+1/m}\), we obtain \eqref{eq:elliptic-estimate}.
\end{proof}

Finally, returning to the original variable \(u=\Psi^{-1}(v)\), we obtain the following corollary, which will be used in Section~\ref{Sec:Para} to construct the sequence of approximate solutions.

\begin{corollary}
	\label{EpiCoco}
	Let \(g\in L^{m+1}(M,\mu)\). Then there exists a unique \(u\) such that
	\[
	u\in L^{m+1}(M,\mu),\qquad \Psi(u)\in\mathcal{F}_e,
	\]
	and \(u\) is the unique weak solution of \eqref{eq:elliptic-u}, i.e.,
	\begin{equation}
		\int_M u\,\phi\,d\mu + \tau\,\mathcal{E}(\Psi(u),\phi) = \int_M g\,\phi\,d\mu \qquad \forall \phi\in V.
		\label{eq:corollary-weak}
	\end{equation}
	Moreover,
	\begin{equation}
		\|u\|_{L^{m+1}}^{m+1} + \tau(m+1)\mathcal{E}(\Psi(u)) \le \|g\|_{L^{m+1}}^{m+1},
		\label{eq:corollary-estimate}
	\end{equation}
	and $u\geq 0$ $\mu$-a.e. provided $g\geq 0$ $\mu$-a.e..
\end{corollary}

\begin{proof}
	This follows directly from Lemma~\ref{lem:elliptic_existence}, Corollary~\ref{prop:nonnega}, Lemma~\ref{lem:elliptic_estimates}, and the change of variables \(u=\Psi^{-1}(v)\).
\end{proof}

\section{Well-posedness of the Cauchy problem \eqref{eq:cauchy}}
\label{Sec:Para}

In this section, we prove the main theorem stated in Section~\ref{sec:Intro}, namely the existence and uniqueness of weak solutions to the Cauchy problem \eqref{eq:cauchy}. The argument is carried out via the Rothe method: we construct a sequence of piecewise constant in time approximations using the elliptic solvers from Section~\ref{Sec:Epi}, establish uniform a priori estimates, pass to a weak limit, and then identify the limit via a monotonicity argument.

\subsection{Construction of approximate solutions and uniform estimates}
\label{subsec:uvw}

For each \(n\ge1\), starting from \(u_{n,0}:=u_0\), we recursively apply Corollary~\ref{EpiCoco} with \(\tau=T/n\) and \(g=u_{n,k-1}\) to obtain a sequence \(\{u_{n,k}\}_{k=0}^{n}\subset L^{m+1}(M,\mu)\) such that, for each \(1\le k\le n\),

\begin{enumerate}
	\item[(i)] \(\Psi(u_{n,k})\in\mathcal{F}_e\), and \(u_{n,k}\) is the unique weak solution of
	\[
	u_{n,k}-\frac{T}{n}\mathcal{L}(\Psi(u_{n,k}))=u_{n,k-1}
	\]
	in the sense of Corollary~\ref{EpiCoco};
	
	\item[(ii)]
	\begin{equation}
		\|u_{n,k}\|_{L^{m+1}}^{m+1} + \frac{T}{n}(m+1)\mathcal{E}(\Psi(u_{n,k}))
		\le \|u_{n,k-1}\|_{L^{m+1}}^{m+1}.
		\label{eq:energy-step}
	\end{equation}
\end{enumerate}

Moreover, if \(u_0\ge0\) \(\mu\)-a.e., then \(u_{n,k}\ge0\) \(\mu\)-a.e. for each \(1\le k\le n\).

For each \(n\ge1\), define the piecewise constant approximate solution \(u_n\) as follows. For every \(t\in(0,T]\), let \(k\in\{1,\dots,n\}\) be the unique index such that \(t\in((k-1)T/n,kT/n]\). Then, for \(\mu\)-a.e. \(x\in M\),
\[
u_n(t,x):=u_{n,k}(x).
\]
At \(t=0\), set \(u_n(0,\cdot)\equiv u_0\).

Based on \(u_n\), we define
\[
v_n(t,x):=\Psi(u_n(t,x)),\qquad
w_n(t,x):=\frac{u_{n,k}-u_{n,k-1}}{T/n}
\quad\text{for } t\in((k-1)T/n,kT/n],
\]
with \(v_n(0,\cdot)\equiv0\) and \(w_n(0,\cdot)\equiv0\).

Equivalently, for \(t\in[0,T]\) and \(\mu\)-a.e. \(x\in M\),
\begin{align}
	u_n(t,x)&=\sum_{k=1}^n u_{n,k}(x)\mathbf{1}_{((k-1)T/n,kT/n]}(t)+u_0(x)\mathbf{1}_{\{0\}}(t), \label{uN} \\
	v_n(t,x)&=\sum_{k=1}^n \Psi(u_{n,k}(x))\mathbf{1}_{((k-1)T/n,kT/n]}(t), \label{vN} \\
	w_n(t,x)&=\sum_{k=1}^n \frac{u_{n,k}-u_{n,k-1}}{T/n}\mathbf{1}_{((k-1)T/n,kT/n]}(t). \label{wN}
\end{align}
In particular, \(v_n(t,\cdot)=\Psi(u_n(t,\cdot))\) for all \(0<t\le T\).

We now establish uniform bounds for the approximate solutions \(u_n\) and the associated quantities \(v_n\) and \(w_n\), independent of \(n\).

\begin{proposition}[Uniform bounds]
	\label{ParaUB}
	The following statements hold.
	
	\begin{enumerate}
		\item[(i)] For each \(n\ge1\), \(u_n\in L^\infty(0,T;L^{m+1})\subset L^2(0,T;L^{m+1})\). Moreover,
		\begin{align}
			\sup_{n\ge1}\|u_n\|_{L^\infty(0,T;L^{m+1})} &\le \|u_0\|_{L^{m+1}}, \label{ustNUB} \\
			\sup_{n\ge1}\|u_n\|_{L^2(0,T;L^{m+1})} &\le \sqrt{T}\|u_0\|_{L^{m+1}}. \label{uNUB}
		\end{align}
		
		\item[(ii)] For each \(n\ge1\), \(v_n\in L^2(0,T;V)\). Moreover,
		\begin{equation}
			\sup_{n\ge1}\|v_n\|_{L^2(0,T;V)}
			\le \sqrt{T\|u_0\|_{L^{m+1}}^{2m} + \frac{\|u_0\|_{L^{m+1}}^{m+1}}{m+1}}
			=: C_*(m,T,\|u_0\|_{L^{m+1}}). \label{vNUB}
		\end{equation}
		
		\item[(iii)] For each \(n\ge1\), \(w_n\in L^2(0,T;V^*)\). Moreover,
		\begin{equation}
			\sup_{n\ge1}\|w_n\|_{L^2(0,T;V^*)}
			\le \sqrt{\frac{\|u_0\|_{L^{m+1}}^{m+1}}{m+1}}. \label{wNUB}
		\end{equation}
	\end{enumerate}
\end{proposition}

\begin{proof}
	We first prove (i). For each \(n\ge1\), it follows from \eqref{eq:energy-step} that, for each \(1\le k\le n\),
	\[
	\|u_{n,k}\|_{L^{m+1}}\le \|u_{n,k-1}\|_{L^{m+1}}.
	\]
	Hence
	\[
	\esup_{t\in[0,T]}\|u_n(t)\|_{L^{m+1}}
	= \max_{0\le k\le n}\|u_{n,k}\|_{L^{m+1}}
	\le \|u_{n,0}\|_{L^{m+1}}
	= \|u_0\|_{L^{m+1}},
	\]
	which gives \(u_n\in L^\infty(0,T;L^{m+1})\) and proves \eqref{ustNUB}. The embedding \(L^\infty(0,T;L^{m+1})\subset L^2(0,T;L^{m+1})\) is immediate.
	
	Furthermore, from the above estimate,
	\[
	\|u_n\|_{L^2(0,T;L^{m+1})}^2
	= \int_0^T \|u_n(t)\|_{L^{m+1}}^2\,dt
	\le T\|u_0\|_{L^{m+1}}^2,
	\]
	which proves \eqref{uNUB}.
	
	We now prove (ii). Since \(|v_n(t,\cdot)|=|u_n(t,\cdot)|^m\) for every \(0<t\le T\), we have
	\begin{align}
		\int_0^T \|v_n(t)\|_{L^{1+1/m}}^2\,dt
		&= \int_0^T \left(\int_M |v_n(t)|^{1+1/m}\,d\mu\right)^{\frac{2}{1+1/m}}\,dt \notag \\
		&= \int_0^T \left(\int_M |u_n(t)|^{1+m}\,d\mu\right)^{\frac{2m}{1+m}}\,dt \notag \\
		&= \int_0^T \|u_n(t)\|_{L^{m+1}}^{2m}\,dt
		\le T\|u_0\|_{L^{m+1}}^{2m}. \label{eq:v-Lp-est}
	\end{align}
	On the other hand, by the definition of \(v_n\) in \eqref{vN} and using \eqref{eq:energy-step},
	\begin{equation}
		\int_0^T \mathcal{E}(v_n(t))\,dt
		= \frac{T}{n}\sum_{k=1}^n \mathcal{E}(\Psi(u_{n,k}))
		\le \frac{1}{m+1}\sum_{k=1}^n \left(\|u_{n,k-1}\|_{L^{m+1}}^{m+1}-\|u_{n,k}\|_{L^{m+1}}^{m+1}\right)
		\le \frac{\|u_0\|_{L^{m+1}}^{m+1}}{m+1}. \label{eq:esumu0}
	\end{equation}
	Combining \eqref{eq:v-Lp-est} and \eqref{eq:esumu0}, we obtain
	\[
	\int_0^T \|v_n(t)\|_V^2\,dt
	= \int_0^T \|v_n(t)\|_{L^{1+1/m}}^2\,dt + \int_0^T \mathcal{E}(v_n(t))\,dt
	\le T\|u_0\|_{L^{m+1}}^{2m} + \frac{\|u_0\|_{L^{m+1}}^{m+1}}{m+1},
	\]
	which shows that \(v_n\in L^2(0,T;V)\) and proves \eqref{vNUB}.
	
	We finally prove (iii). From \eqref{eq:elliptic-weak}, for each \(1\le k\le n\) and any \(\phi\in V\),
	\[
	\left|\left\langle \frac{u_{n,k}-u_{n,k-1}}{T/n}, \phi\right\rangle_V\right|
	= |\mathcal{E}(\Psi(u_{n,k}),\phi)|
	\le \sqrt{\mathcal{E}(\Psi(u_{n,k}))}\sqrt{\mathcal{E}(\phi)}
	\le \sqrt{\mathcal{E}(\Psi(u_{n,k}))}\|\phi\|_V,
	\]
	hence
	\begin{equation}
		\left\|\frac{u_{n,k}-u_{n,k-1}}{T/n}\right\|_{V^*}
		\le \sqrt{\mathcal{E}(\Psi(u_{n,k}))}. \label{eq:vastnorm}
	\end{equation}
	Using \eqref{eq:esumu0}, we get
	\[
	\int_0^T \|w_n(t)\|_{V^*}^2\,dt
	= \sum_{k=1}^n \int_{(k-1)T/n}^{kT/n}
	\left\|\frac{u_{n,k}-u_{n,k-1}}{T/n}\right\|_{V^*}^2\,dt
	\le \frac{T}{n}\sum_{k=1}^n \mathcal{E}(\Psi(u_{n,k}))
	\le \frac{\|u_0\|_{L^{m+1}}^{m+1}}{m+1}.
	\]
	Thus \(w_n\in L^2(0,T;V^*)\) and \eqref{wNUB} holds.
\end{proof}

Since \(V\) is reflexive (Corollary~\ref{cor:Vq_reflexive}) and \(L^{m+1}(M,\mu)\) is reflexive, it follows from the Phillips theorem for vector-valued \(L^p\)-spaces (see, e.g., \cite[Corollary 2, \S IV]{1977Vector}) that both \(L^2(0,T;V)\) and \(L^2(0,T;V^*)\) are reflexive. Moreover, the separability of \(L^{m+1}(M,\mu)\) implies
\[
L^\infty(0,T;L^{m+1}) = \left(L^1(0,T;L^{1+1/m})\right)^*
\]
(see, e.g., \cite[Problem~23.12d]{Zeidler1990IIA}). Consequently, the uniform bounds from Proposition~\ref{ParaUB} allow us to apply the Banach--Alaoglu theorem for weak-* convergence and the Eberlein--\v{S}mulian theorem for weak convergence, yielding a subsequence \(\{n_j\}\subset\mathbb{N}\) such that $n_j\uparrow\infty$ as $j\to\infty$, and
\begin{align}
	u_{n_j} &\overset{*}{\rightharpoonup} u^* &&\text{in } L^\infty(0,T;L^{m+1}), \label{eq:uast-weak} \\
	u_{n_j} &\rightharpoonup u &&\text{in } L^2(0,T;L^{m+1}), \label{eq:u-weak} \\
	v_{n_j} &\rightharpoonup v &&\text{in } L^2(0,T;V), \label{eq:v-weak} \\
	w_{n_j} &\rightharpoonup w &&\text{in } L^2(0,T;V^*). \label{eq:w-weak}
\end{align}
as $j\to\infty$.

Since the same subsequence converges weakly in \(L^2(0,T;L^{m+1})\) and weak-* in \(L^\infty(0,T;L^{m+1})\), the two limits agree a.e. on \((0,T)\times M\). Hence \(u(t,\cdot)=u^*(t,\cdot)\) for a.e. \(t\in(0,T)\), and consequently \(u\in L^\infty(0,T;L^{m+1})\).

\subsection{Derivation of the weak formulation}

We first establish a discrete energy identity that will serve as the starting point for the passage to the limit.

For \(\varphi\in C_c^1([0,T);V)\) and \(n\ge1\), define
\[
\varphi_{n,k}:=\varphi\left(\frac{(k-1)T}{n},\cdot\right)\quad\text{for each }1\le k\le n,
\]
and set \(\varphi_{n,n+1}\equiv0\). Define the piecewise constant approximation of \(\varphi\) by
\[
\varphi_n(t,x):=\sum_{k=1}^n \varphi_{n,k}(x)\mathbf{1}_{\left[\frac{(k-1)T}{n},\frac{kT}{n}\right)}(t)+0\cdot\mathbf{1}_{\{T\}}(t)
\]
for \(0\le t\le T\) and \(\mu\)-a.e. \(x\in M\).

The following standard estimate will be used to pass to the limit in the discrete energy identity.

\begin{lemma}
	\label{ParaTestStrong}
	For every \(\varphi\in C_c^1([0,T);V)\),
	\begin{equation}
		\|\varphi_n-\varphi\|_{L^2(0,T;V)}
		\le \frac{T}{\sqrt{n}}\|\partial_t\varphi\|_{L^2(0,T;V)}. \label{vphiNvphi}
	\end{equation}
	In particular, \(\varphi_n\to\varphi\) strongly in \(L^2(0,T;V)\) as \(n\to\infty\).
\end{lemma}

\begin{proof}
	From the definition of \(\varphi_n\),
	\begin{align}
		\|\varphi_n-\varphi\|_{L^2(0,T;V)}^2
		&= \sum_{k=1}^n \int_{(k-1)T/n}^{kT/n}
		\left\|\varphi\left(\frac{(k-1)T}{n}\right)-\varphi(t)\right\|_V^2 dt.
		\label{vphit01}
	\end{align}
	By \cite[Theorem~2(ii), \S~5.9.2]{Evans.2010.749},
	\[
	\varphi\left(\frac{(k-1)T}{n}\right)-\varphi(t)
	= -\int_{(k-1)T/n}^{t}\partial_t\varphi(s)\,ds.
	\]
	Thus, by the Bochner theorem (see \cite[Theorem~8, Appendix~E.5]{Evans.2010.749}) and the Cauchy--Schwarz inequality, for any \(t\in[(k-1)T/n,kT/n]\),
	\begin{align}
		\left\|\varphi\left(\frac{(k-1)T}{n}\right)-\varphi(t)\right\|_V^2
		&= \left\|\int_{(k-1)T/n}^{t}\partial_t\varphi(s)\,ds\right\|_V^2 \notag \\
		&\le \left(\int_{(k-1)T/n}^{t}\|\partial_t\varphi(s)\|_V\,ds\right)^2 \notag \\
		&\le \left(t-\frac{(k-1)T}{n}\right)
		\int_{(k-1)T/n}^{t}\|\partial_t\varphi(s)\|_V^2\,ds \notag \\
		&\le \frac{T}{n}\|\partial_t\varphi\|_{L^2(0,T;V)}^2. \label{vphit02}
	\end{align}
	Substituting \eqref{vphit02} into \eqref{vphit01} yields
	\[
	\|\varphi_n-\varphi\|_{L^2(0,T;V)}^2
	\le \sum_{k=1}^n \int_{(k-1)T/n}^{kT/n}
	\frac{T}{n}\|\partial_t\varphi\|_{L^2(0,T;V)}^2\,dt
	= \frac{T^2}{n}\|\partial_t\varphi\|_{L^2(0,T;V)}^2,
	\]
	which proves \eqref{vphiNvphi}.
\end{proof}

The following discrete energy identity is the core of the Rothe method.

\begin{proposition}[Discrete energy identity]
	\label{ParaWeak}
	For every \(\varphi\in C_c^1([0,T);V)\) and every \(n\ge1\),
	\begin{equation}
		-\int_0^T\int_M u_n\,\partial_t\varphi\,d\mu\,dt
		+\int_0^T \mathcal{E}(v_n(t),\varphi_n(t))\,dt
		= \int_M u_0\,\varphi(0)\,d\mu. \label{eq:discrete-id}
	\end{equation}
\end{proposition}

\begin{proof}
	For each \(1\le k\le n\), taking \(\phi=\varphi_{n,k}\) in the weak formulation for \(u_{n,k}\) (see Corollary~\ref{EpiCoco}), we obtain
	\begin{equation}
		\int_M (u_{n,k}-u_{n,k-1})\varphi_{n,k}\,d\mu
		+ \frac{T}{n}\mathcal{E}(\Psi(u_{n,k}),\varphi_{n,k}) = 0.
		\label{eq:discrete-step}
	\end{equation}
	Summing \eqref{eq:discrete-step} over \(k=1,\dots,n\) gives
	\begin{equation}
		\sum_{k=1}^n \int_M u_{n,k}\varphi_{n,k}\,d\mu
		-\sum_{k=1}^n \int_M u_{n,k-1}\varphi_{n,k}\,d\mu
		+\frac{T}{n}\sum_{k=1}^n \mathcal{E}(\Psi(u_{n,k}),\varphi_{n,k}) = 0.
		\label{eq:discrete-sum}
	\end{equation}
	
	We now rewrite the time-integral terms. First, by the definition of \(v_n\) in \eqref{vN},
	\begin{equation} \label{eq:energy-sum}
		\int_0^T \mathcal{E}(v_n(t),\varphi_n(t))\,dt
		= \sum_{k=1}^n \int_{(k-1)T/n}^{kT/n}
		\mathcal{E}(\Psi(u_{n,k}),\varphi_{n,k})\,dt
		= \frac{T}{n}\sum_{k=1}^n \mathcal{E}(\Psi(u_{n,k}),\varphi_{n,k}).
	\end{equation}
	Second,
	\begin{align}
		-\int_0^T\int_M u_n(t)\,\partial_t\varphi(t)\,d\mu\,dt
		&= -\sum_{k=1}^n \int_{(k-1)T/n}^{kT/n}
		\int_M u_{n,k}\,\partial_t\varphi(t)\,d\mu\,dt \notag \\
		&= -\sum_{k=1}^n \int_M u_{n,k}
		\left(\varphi\left(\frac{kT}{n},\cdot\right)-\varphi\left(\frac{(k-1)T}{n},\cdot\right)\right)d\mu \notag \\
		&= -\sum_{k=1}^n \int_M u_{n,k}(\varphi_{n,k+1}-\varphi_{n,k})\,d\mu \notag \\
		&= \sum_{k=1}^n \int_M u_{n,k}\varphi_{n,k}\,d\mu
		-\sum_{k=1}^n \int_M u_{n,k}\varphi_{n,k+1}\,d\mu, \label{eq:time-sum}
	\end{align}
	where we used \(\varphi_{n,n+1}\equiv0\).
	
	Adding \eqref{eq:energy-sum} and \eqref{eq:time-sum}, and comparing with \eqref{eq:discrete-sum}, we get
	\begin{align}
		& -\int_0^T\int_M u_n\,\partial_t\varphi\,d\mu\,dt
		+\int_0^T \mathcal{E}(v_n(t),\varphi_n(t))\,dt \notag \\
		&= \sum_{k=1}^n \int_M u_{n,k}\varphi_{n,k}\,d\mu
		-\sum_{k=1}^n \int_M u_{n,k}\varphi_{n,k+1}\,d\mu
		+\frac{T}{n}\sum_{k=1}^n \mathcal{E}(\Psi(u_{n,k}),\varphi_{n,k}) \notag \\
		&= \Bigg(\sum_{k=1}^n \int_M u_{n,k}\varphi_{n,k}\,d\mu
		-\sum_{k=1}^n \int_M u_{n,k-1}\varphi_{n,k}\,d\mu
		+\frac{T}{n}\sum_{k=1}^n \mathcal{E}(\Psi(u_{n,k}),\varphi_{n,k})\Bigg) \notag \\
		&\qquad + \Bigg(\sum_{k=1}^n \int_M u_{n,k-1}\varphi_{n,k}\,d\mu
		-\sum_{k=1}^n \int_M u_{n,k}\varphi_{n,k+1}\,d\mu\Bigg) \notag \\
		&= 0 + \int_M u_0\,\varphi_{n,1}\,d\mu - \int_M u_n\,\varphi_{n,n+1}\,d\mu \notag \\
		&= \int_M u_0\,\varphi(0,\cdot)\,d\mu,
	\end{align}
	where the last equality follows from \(\varphi_{n,1}=\varphi(0,\cdot)\) and \(\varphi_{n,n+1}\equiv0\). This proves \eqref{eq:discrete-id}.
\end{proof}

\begin{proposition}[Weak formulation on smooth test functions]
	\label{prop:weak-form}
	Let \(u\) and \(v\) be the weak limits obtained in \eqref{eq:u-weak} and \eqref{eq:v-weak}, respectively. Then, for every \(\varphi\in C_c^1([0,T);V)\),
	\begin{equation}
		-\int_0^T\int_M u\,\partial_t\varphi\,d\mu\,dt
		+\int_0^T \mathcal{E}(v(t),\varphi(t))\,dt
		= \int_M u_0\,\varphi(0)\,d\mu. \label{eq:weak-form}
	\end{equation}
\end{proposition}

\begin{proof}
	Let \(\{n_j\}_{j\ge1}\subset\mathbb{N}\) be the subsequence for which both \eqref{eq:u-weak} and \eqref{eq:v-weak} hold.
	
	Since \eqref{eq:discrete-id} holds for each \(n\ge1\), for any \(j\ge1\),
	\begin{equation}
		-\int_0^T\int_M u_{n_j}\,\partial_t\varphi\,d\mu\,dt
		+\int_0^T \mathcal{E}(v_{n_j}(t),\varphi_{n_j}(t))\,dt
		= \int_M u_0\,\varphi(0)\,d\mu. \label{eq:nj-weakform}
	\end{equation}
	
	Since \(\varphi\in C_c^1([0,T);V)\), we have
	\[
	\partial_t\varphi\in C_c([0,T);V)\subset L^2(0,T;V)\subset L^2(0,T;L^{1+1/m}).
	\]
	Thus, for any \(\tilde{u}\in L^2(0,T;L^{m+1})\), by H\"older's inequality,
	\[
	\left|\int_0^T\int_M \tilde{u}\,\partial_t\varphi\,d\mu\,dt\right|
	\le \|\tilde{u}\|_{L^2(0,T;L^{m+1})}
	\|\partial_t\varphi\|_{L^2(0,T;L^{1+1/m})}.
	\]
	Hence the linear functional
	\[
	\tilde{u}\mapsto \int_0^T\int_M \tilde{u}\,\partial_t\varphi\,d\mu\,dt
	\]
	is bounded on \(L^2(0,T;L^{m+1})\). Therefore,
	\begin{equation}
		\lim_{j\to\infty}\int_0^T\int_M u_{n_j}\,\partial_t\varphi\,d\mu\,dt
		= \int_0^T\int_M u\,\partial_t\varphi\,d\mu\,dt. \label{eq:nj1-weak}
	\end{equation}
	
	We now claim that
	\begin{equation}
		\lim_{j\to\infty}\int_0^T \mathcal{E}(v_{n_j}(t),\varphi_{n_j}(t))\,dt
		= \int_0^T \mathcal{E}(v(t),\varphi(t))\,dt. \label{eq:nj2-weak}
	\end{equation}
	Indeed,
	\[
	\int_0^T \mathcal{E}(v_{n_j}(t),\varphi_{n_j}(t))\,dt
	= \int_0^T \mathcal{E}(v_{n_j}(t),\varphi(t))\,dt
	+ \int_0^T \mathcal{E}(v_{n_j}(t),\varphi_{n_j}(t)-\varphi(t))\,dt.
	\]
	For the first term at the right hand side of the equality above, note that for any \(\tilde{v}\in L^2(0,T;V)\),
	\[
	\left|\int_0^T \mathcal{E}(\tilde{v}(t),\varphi(t))\,dt\right|
	\le \|\tilde{v}\|_{L^2(0,T;V)}\|\varphi\|_{L^2(0,T;V)},
	\]
	so the linear functional \(\tilde{v}\mapsto \int_0^T \mathcal{E}(\tilde{v}(t),\varphi(t))\,dt\) is bounded on \(L^2(0,T;V)\). Since \(v_{n_j}\rightharpoonup v\) in \(L^2(0,T;V)\),
	\begin{equation}
		\lim_{j\to\infty}\int_0^T \mathcal{E}(v_{n_j}(t),\varphi(t))\,dt
		= \int_0^T \mathcal{E}(v(t),\varphi(t))\,dt. \label{njv}
	\end{equation}
	For the second term, by H\"older's inequality, \eqref{eq:esumu0}, and Lemma~\ref{ParaTestStrong},
	\begin{align}
		&\left|\int_0^T \mathcal{E}(v_{n_j}(t),\varphi_{n_j}(t)-\varphi(t))\,dt\right| \notag \\
		&\le \int_0^T \sqrt{\mathcal{E}(v_{n_j}(t))}
		\|\varphi_{n_j}(t)-\varphi(t)\|_V\,dt \notag \\
		&\le \left(\int_0^T \mathcal{E}(v_{n_j}(t))\,dt\right)^{1/2}
		\|\varphi_{n_j}-\varphi\|_{L^2(0,T;V)} \notag \\
		&\le \sqrt{\frac{\|u_0\|_{L^{m+1}}^{m+1}}{m+1}}
		\cdot \frac{T}{\sqrt{n_j}}
		\|\partial_t\varphi\|_{L^2(0,T;V)}. \label{eq:Evvarphi}
	\end{align}
	Thus
	\[
	\lim_{j\to\infty}\int_0^T \mathcal{E}(v_{n_j}(t),\varphi_{n_j}(t)-\varphi(t))\,dt = 0.
	\]
	Combining this with \eqref{njv} proves \eqref{eq:nj2-weak}.
	
	Finally, letting \(j\to\infty\) in \eqref{eq:nj-weakform} and using \eqref{eq:nj1-weak} and \eqref{eq:nj2-weak}, we obtain \eqref{eq:weak-form}.
\end{proof}

\subsection{Regularity improvement of $u$ in time}

We now establish that the limit function \(u\) satisfies the time regularity required in Definition~\ref{def:weaksol}, namely \(u\in W^{1,2}(0,T;V^*)\). This property is independent of the derivation of the weak formulation and serves only to verify that the limiting object belongs to the correct function space. The following proposition identifies the weak time derivative of \(u\) with the limit \(w\) from \eqref{eq:w-weak}.

\begin{proposition}
	\label{prop:wd}
	The function \(w\) is the weak derivative of \(u\); that is, for any \(f\in C_c^\infty((0,T))\) and any \(v\in V\),
	\begin{equation} \label{eq:wd}
		\int_0^T \langle u(t), f'(t)v\rangle_V\,dt
		= -\int_0^T \langle w(t), f(t)v\rangle_V\,dt.
	\end{equation}
\end{proposition}

\begin{proof}
	For each \(n\ge1\) and each \(0\le k\le n\), set
	\[
	s_{n,k}:=\langle u_{n,k}, v\rangle_V,\qquad f_k:=f\left(\frac{kT}{n}\right),
	\]
	and note that \(f_0=f_n=0\). Then
	\begin{align}
		\int_0^T \langle u_n(t), f'(t)v\rangle_V\,dt
		&= \sum_{k=1}^n \int_{(k-1)T/n}^{kT/n}
		\langle u_{n,k}, f'(t)v\rangle_V\,dt \notag \\
		&= \sum_{k=1}^n \langle u_{n,k}, v\rangle_V
		\int_{(k-1)T/n}^{kT/n} f'(t)\,dt \notag \\
		&= \sum_{k=1}^n s_{n,k}\left(f_k-f_{k-1}\right)
		= \sum_{k=1}^{n-1} s_{n,k}f_k - \sum_{k=1}^n s_{n,k}f_{k-1}. \label{eq:wd1}
	\end{align}
	On the other hand,
	\begin{align}
		-\int_0^T \langle w_n(t), f(t)v\rangle_V\,dt
		&= -\sum_{k=1}^n \int_{(k-1)T/n}^{kT/n}
		\left\langle \frac{u_{n,k}-u_{n,k-1}}{T/n}, f(t)v\right\rangle_V\,dt \notag \\
		&= -\sum_{k=1}^n (s_{n,k}-s_{n,k-1})
		\int_{(k-1)T/n}^{kT/n} \frac{f(t)}{T/n}\,dt \notag \\
		&= -\sum_{k=1}^n (s_{n,k}-s_{n,k-1})
		\left(f_{k-1} + \int_{(k-1)T/n}^{kT/n} \frac{f(t)-f_{k-1}}{T/n}\,dt\right) \notag \\
		&= \sum_{k=1}^n (s_{n,k-1}-s_{n,k})f_{k-1}
		+ \sum_{k=1}^n \frac{s_{n,k-1}-s_{n,k}}{T/n}
		\int_{(k-1)T/n}^{kT/n}\int_{(k-1)T/n}^t f'(s)\,ds\,dt \notag \\
		&=: I_1(n) + I_2(n). \label{eq:wd2}
	\end{align}
	Since \(f_0=f_n=0\), we have
	\[
	I_1(n)=\sum_{k=1}^n (s_{n,k-1}-s_{n,k})f_{k-1}
	= \sum_{k=1}^{n-1} s_{n,k}f_k - \sum_{k=1}^n s_{n,k}f_{k-1}
	= \int_0^T \langle u_n(t), f'(t)v\rangle_V\,dt.
	\]
	Comparing this with \eqref{eq:wd1} and \eqref{eq:wd2}, we obtain
	\begin{equation} \label{eq:wdlim}
		\int_0^T \langle u_n(t), f'(t)v\rangle_V\,dt + I_2(n)
		= -\int_0^T \langle w_n(t), f(t)v\rangle_V\,dt,
	\end{equation}
	where
	\[
	I_2(n):=\sum_{k=1}^n \frac{s_{n,k-1}-s_{n,k}}{T/n}
	\int_{(k-1)T/n}^{kT/n}\int_{(k-1)T/n}^t f'(s)\,ds\,dt.
	\]
	
	We now show that \(I_2(n)\to0\) as \(n\to\infty\). Indeed,
	\begin{align}
		|I_2(n)|
		&\le \sum_{k=1}^n \left|\frac{s_{n,k-1}-s_{n,k}}{T/n}\right|
		\int_{(k-1)T/n}^{kT/n}\int_{(k-1)T/n}^t |f'(s)|\,ds\,dt \notag \\
		&= \sum_{k=1}^n \left|\left\langle \frac{u_{n,k-1}-u_{n,k}}{T/n}, v\right\rangle_V\right|
		\int_{(k-1)T/n}^{kT/n}\int_{(k-1)T/n}^t |f'(s)|\,ds\,dt \notag \\
		&\le \sum_{k=1}^n \left\|\frac{u_{n,k-1}-u_{n,k}}{T/n}\right\|_{V^*}\|v\|_V
		\left(\|f'\|_{C[0,T]}\frac{T^2}{n^2}\right) \notag \\
		&\le \frac{T^2\|v\|_V\|f'\|_{C[0,T]}}{n^2}
		\sum_{k=1}^n \sqrt{\mathcal{E}(\Psi(u_{n,k}))}, \label{eq:wdI2}
	\end{align}
	where we used \eqref{eq:vastnorm} in the last inequality.
	
	By Cauchy--Schwarz inequality and \eqref{eq:esumu0},
	\[
	\sum_{k=1}^n \sqrt{\mathcal{E}(\Psi(u_{n,k}))}
	\le \sqrt{n} \left(\sum_{k=1}^n \mathcal{E}(\Psi(u_{n,k}))\right)^{1/2}
	\le \sqrt{n}\sqrt{\frac{n}{T}\cdot \frac{\|u_0\|_{L^{m+1}}^{m+1}}{m+1}}.
	\]
	Substituting this into \eqref{eq:wdI2}, we get
	\[
	|I_2(n)|
	\le \frac{T^{3/2}\|v\|_V\|f'\|_{C[0,T]}}{n}
	\sqrt{\frac{\|u_0\|_{L^{m+1}}^{m+1}}{m+1}},
	\]
	so \(I_2(n)\to0\) as \(n\to\infty\).
	
	Since \eqref{eq:wdlim} holds for every \(n_j\), \(j\ge1\), we let \(j\to\infty\). Using \(u_{n_j}\rightharpoonup u\) in \(L^2(0,T;L^{m+1})\), \(w_{n_j}\rightharpoonup w\) in \(L^2(0,T;V^*)\), and \(I_2(n_j)\to0\), we obtain \eqref{eq:wd} from \eqref{eq:wdlim}.
\end{proof}

\begin{corollary}
	\label{cor;up}
	Let \(u\) be the weak limit in \eqref{eq:u-weak}. Then
	\[
	u\in L^\infty(0,T;L^{m+1})\cap W^{1,2}(0,T;V^*).
	\]
\end{corollary}

\begin{proof}
	It is immediate from \eqref{eq:uast-weak} and \eqref{eq:u-weak} that
	\[
	u\in L^\infty(0,T;L^{m+1})\cap L^2(0,T;L^{m+1})
	= L^\infty(0,T;L^{m+1}).
	\]
	
	Since \(V\hookrightarrow L^{1+1/m}\), we have \(L^{m+1}\hookrightarrow V^*\), and hence
	\[
	L^2(0,T;L^{m+1})\hookrightarrow L^2(0,T;V^*).
	\]
	By Proposition~\ref{prop:wd}, the weak derivative of \(u\) exists and lies in \(L^2(0,T;V^*)\). Therefore \(u\in W^{1,2}(0,T;V^*)\).
\end{proof}

\subsection{Identification of the limit via the Minty trick}
\label{sub:unicom}

In the preceding subsections, we have obtained the weak formulation \eqref{eq:weak-form} (Proposition~\ref{prop:weak-form}) and established the required time regularity of \(u\). However, \eqref{eq:weak-form} does not yet identify \(v\) with \(\Psi(u)=|u|^{m-1}u\). To prove this identification, we employ the classical Minty trick. As a first step, we extend the weak formulation to a larger class of test functions.

Set
\[
\mathcal{H}:= \{\varphi\in W^{1,2}(0,T;V) : \varphi(T,\cdot)=0 \ \mu\text{-a.e.}\}.
\]
Since \(C_c^1([0,T);V)\) is dense in \(\mathcal{H}\) by the standard density argument for Bochner--Sobolev spaces, the following result follows immediately.

\begin{proposition}
	\label{prop:ext}
	The identity \eqref{eq:weak-form} holds for every \(\varphi\in\mathcal{H}\).
\end{proposition}

\begin{proof}
	Fix \(\varphi\in\mathcal{H}\) and let \(\{\varphi_{(n)}\}_{n\ge1}\subset C_c^1([0,T);V)\) be a sequence approximating \(\varphi\) in \(\mathcal{H}\). By Proposition~\ref{prop:weak-form}, for each \(n\ge1\),
	\begin{equation}
		-\int_0^T\int_M u\,\partial_t\varphi_{(n)}\,d\mu\,dt
		+\int_0^T \mathcal{E}(v(t,\cdot),\varphi_{(n)}(t,\cdot))\,dt
		= \int_M u_0\,\varphi_{(n)}(0,\cdot)\,d\mu. \label{eq:varphi_n}
	\end{equation}
	Since \(\varphi\cup\{\varphi_{(n)}\}_{n\ge1}\subset W^{1,2}(0,T;V)\), it follows from \cite[Theorem 2, \S~5.9.2]{Evans.2010.749} that there exists a constant \(C>0\) such that, for every \(n\ge1\),
	\begin{equation}
		\|\varphi_{(n)}(0)-\varphi(0)\|_V
		\le \max_{t\in[0,T]}\|\varphi_{(n)}(t)-\varphi(t)\|_V
		\le C\|\varphi_{(n)}-\varphi\|_{W^{1,2}(0,T;V)}. \label{eq:varphi_nn}
	\end{equation}
	On the other hand, the linear functional
	\[
	\varphi \mapsto -\int_0^T\int_M u\,\partial_t\varphi\,d\mu\,dt
	+\int_0^T \mathcal{E}(v(t,\cdot),\varphi(t,\cdot))\,dt
	\]
	is bounded on \(W^{1,2}(0,T;V)\). Letting \(n\to\infty\) in \eqref{eq:varphi_n} and using \eqref{eq:varphi_nn}, we obtain the desired result.
\end{proof}

Before applying the classical Minty trick to prove that \(v=\Psi(u)\) in the sense that \(v(t,\cdot)=\Psi(u(t,\cdot))\) for a.e. \(t\in(0,T)\), we need the following auxiliary inequality. It compares the space-time integral \(\int_0^T\int_M uv\,d\mu\,dt\) with the upper limit of the discrete integrals \(\int_0^T\int_M u_{n_j}v_{n_j}\,d\mu\,dt\).

\begin{proposition}
	\label{prop:liminf}
	Let \(\{u_{n_j}\}_{j\ge1}\) and \(\{v_{n_j}\}_{j\ge1}\) be the sequences from \eqref{eq:u-weak} and \eqref{eq:v-weak}, respectively. Then
	\begin{equation}
		\liminf_{j\to\infty}\int_0^T\int_M u_{n_j}v_{n_j}\,d\mu\,dt
		\le \int_0^T\int_M uv\,d\mu\,dt. \label{eq:liminf}
	\end{equation}
\end{proposition}

\begin{proof}
	Set
	\[
	\varphi:=\int_t^T v(s)\,ds.
	\]
	One readily checks that \(\varphi\in\mathcal{H}\) and \(\partial_t\varphi=-v\). Hence, by Proposition~\ref{prop:ext},
	\begin{equation}
		\int_0^T\int_M uv\,d\mu\,dt
		+\int_0^T \mathcal{E}\left(v(t),\int_t^T v(s)\,ds\right)dt
		= \int_M u_0\int_0^T v(s)\,ds\,d\mu. \label{eq:test-1}
	\end{equation}
	By Lemma~\ref{lem:fubini}, Fubini's theorem, and the symmetry of \(\mathcal{E}\),
	\begin{align}
		\int_0^T \mathcal{E}\left(v(t),\int_t^T v(s)\,ds\right)dt
		&= \int_0^T\int_t^T \mathcal{E}(v(t),v(s))\,ds\,dt \notag \\
		&= \frac{1}{2}\int_0^T\int_0^T \mathcal{E}(v(t),v(s))\,ds\,dt \notag \\
		&= \frac{1}{2}\mathcal{E}\left(\int_0^T v(t)\,dt\right).
	\end{align}
	Substituting this into \eqref{eq:test-1} yields
	\begin{equation}
		\int_0^T\int_M uv\,d\mu\,dt
		= \int_M u_0\int_0^T v(s)\,ds\,d\mu
		-\frac{1}{2}\mathcal{E}\left(\int_0^T v(t)\,dt\right). \label{eq:test-101}
	\end{equation}
	
	We now estimate \(\int_0^T\int_M u_{n_j}v_{n_j}\,d\mu\,dt\). From \eqref{eq:nj-weakform}, we have
	\begin{align}
		&-\int_0^T\int_M u_{n_j}\,\partial_t\varphi\,d\mu\,dt
		+\int_0^T \mathcal{E}(v_{n_j}(t),\varphi(t))\,dt \notag \\
		&\le -\int_0^T\int_M u_{n_j}\,\partial_t\varphi\,d\mu\,dt
		+\int_0^T \mathcal{E}(v_{n_j}(t),\varphi_{n_j}(t))\,dt 
		+\left|\int_0^T \mathcal{E}(v_{n_j}(t),\varphi_{n_j}(t)-\varphi(t))\,dt\right| \notag \\
		&\le \int_M u_0\,\varphi(0,\cdot)\,d\mu
		+\sqrt{\frac{\|u_0\|_{L^{m+1}}^{m+1}}{m+1}}
		\cdot \frac{T}{\sqrt{n_j}}
		\|\partial_t\varphi\|_{L^2(0,T;V)}. \label{eq:test-2}
	\end{align}
	By the density argument used in the proof of Proposition~\ref{prop:ext}, the estimate \eqref{eq:test-2} holds for every \(\varphi\in\mathcal{H}\). Applying it with
	\[
	\varphi=\int_t^T v_{n_j}(s)\,ds
	\]
	and using \eqref{vNUB}, we obtain
	\begin{align}
		&\int_0^T\int_M u_{n_j}v_{n_j}\,d\mu\,dt
		+\int_0^T \mathcal{E}\left(v_{n_j}(t),\int_t^T v_{n_j}(s)\,ds\right)dt \notag \\
		&\le \int_M u_0\int_0^T v_{n_j}(s)\,ds\,d\mu
		+\sqrt{\frac{\|u_0\|_{L^{m+1}}^{m+1}}{m+1}}
		\cdot \frac{T}{\sqrt{n_j}}
		\|v_{n_j}\|_{L^2(0,T;V)} \notag \\
		&\le \int_M u_0\int_0^T v_{n_j}(s)\,ds\,d\mu
		+\frac{C_\#}{\sqrt{n_j}}, \label{eq:test-201}
	\end{align}
	where
	\[
	C_\#:=C_\#(m,T,\|u_0\|_{L^{m+1}})
	:=\sqrt{\frac{\|u_0\|_{L^{m+1}}^{m+1}}{m+1}}\,T\,C_*(m,T,\|u_0\|_{L^{m+1}})
	\]
	is a nonnegative constant depending only on \(m\), \(T\), and \(\|u_0\|_{L^{m+1}}\).
	
	Again, by Lemma~\ref{lem:fubini}, Fubini's theorem, and the symmetry of \(\mathcal{E}\),
	\[
	\int_0^T \mathcal{E}\left(v_{n_j}(t),\int_t^T v_{n_j}(s)\,ds\right)dt
	= \frac{1}{2}\mathcal{E}\left(\int_0^T v_{n_j}(s)\,ds\right).
	\]
	Thus, for each \(j\ge1\),
	\begin{equation}
		\int_0^T\int_M u_{n_j}v_{n_j}\,d\mu\,dt
		\le \int_M u_0\int_0^T v_{n_j}(s)\,ds\,d\mu
		+\frac{C_\#}{\sqrt{n_j}}
		-\frac{1}{2}\mathcal{E}\left(\int_0^T v_{n_j}(s)\,ds\right). \label{eq:test-202}
	\end{equation}
	Since \(v_{n_j}\rightharpoonup v\) in \(L^2(0,T;V)\), Lemma~\ref{lem:bochner-weak} gives
	\begin{equation}
		\int_0^T v_{n_j}(s)\,ds \rightharpoonup \int_0^T v(s)\,ds \quad\text{in } V,
		\label{eq:intweak}
	\end{equation}
	and hence, by \eqref{eq:Eliminf},
	\begin{equation}
		\mathcal{E}\left(\int_0^T v(s)\,ds\right)
		\le \liminf_{j\to\infty}\mathcal{E}\left(\int_0^T v_{n_j}(s)\,ds\right). \label{eq:lscE}
	\end{equation}
	Letting \(j\to\infty\) in \eqref{eq:test-202} and using \eqref{eq:intweak}, \eqref{eq:lscE}, and \eqref{eq:test-101}, we obtain
	\[
	\limsup_{j\to\infty}\int_0^T\int_M u_{n_j}v_{n_j}\,d\mu\,dt
	\le \int_0^T\int_M uv\,d\mu\,dt,
	\]
	which proves \eqref{eq:liminf}.
\end{proof}

We now complete the proof of Theorem~\ref{thm:main}(i), in which we use the Minty trick to prove the existence.

\begin{proof}[Proof of Theorem~\ref{thm:main}(i)]

Let $u$ be the function obtained in \eqref{eq:u-weak}. It is suffice to show that $u$ is the unique weak solution of \eqref{eq:cauchy} with satisfying \eqref{eq:bound} and \eqref{eq:boundv}. We first prove that \(v(t)=\Psi(u(t))\) for a.e. \(t\in(0,T)\).

	
For every $j\ge1$, since $v_{n_j}(t)=\Psi(u_{n_j}(t))$ for any $0<t\le T$, the strict monotonicity of $\Psi$ implies that, for any $h\in L^\infty(0,T;L^{m+1})$, $\Psi(h)\in L^\infty(0,T;L^{1+1/m})$ and
\[
\int_0^T\int_M (u_{n_j}-h)(v_{n_j}-\Psi(h))\,d\mu\,dt \ge 0,
\]
equivalently,
\begin{equation}
	\int_0^T\int_M u_{n_j}v_{n_j}\,d\mu\,dt
	\ge \int_0^T\int_M h v_{n_j}\,d\mu\,dt
	+\int_0^T\int_M u_{n_j}\Psi(h)\,d\mu\,dt
	-\int_0^T\int_M h\Psi(h)\,d\mu\,dt.
	\label{eq:Minty1}
\end{equation}
Note that, the functionals
\[
\tilde u\mapsto \int_0^T\int_M \tilde u\,\Psi(h)\,d\mu\,dt,\qquad
\tilde v\mapsto \int_0^T\int_M \tilde v\,h\,d\mu\,dt
\]
are bounded on $L^2(0,T;L^{m+1})$ and $L^2(0,T;L^{1+1/m})$, respectively. Letting $j\to\infty$ in \eqref{eq:Minty1} and using \eqref{eq:liminf}, we get
\begin{eqnarray*}
\int_0^T\int_M uv\,d\mu\,dt &\ge&\limsup_{j\to\infty}\int_0^T\int_M u_{n_j}v_{n_j}\,d\mu\,dt\\
&\ge& \int_0^T\int_M h v\,d\mu\,dt
+\int_0^T\int_M u\Psi(h)\,d\mu\,dt
-\int_0^T\int_M h\Psi(h)\,d\mu\,dt,
\end{eqnarray*}
i.e.,
\begin{equation}
	\int_0^T\int_M (u-h)(v-\Psi(h))\,d\mu\,dt \ge 0
	\label{eq:Minty2}
\end{equation}
for every $h\in L^2(0,T;L^{m+1})$.

Taking $h=u\pm\lambda w$ with $\lambda>0$ and $w\in L^\infty(0,T;L^{m+1})$, and letting $\lambda\downarrow0$, we obtain
\begin{equation}\label{eq:minty-infty}
\int_0^T\int_M (v-\Psi(u))w\,d\mu\,dt = 0
\end{equation}
for all $w\in L^\infty(0,T;L^{m+1})$. Noting that $L^\infty(0,T;L^{m+1})$ is dense in $L^2(0,T;L^{m+1})$, it follows that \eqref{eq:minty-infty} also holds for any $w$ in $L^2(0,T;L^{m+1})$, thus implying that $v=\Psi(u)$ a.e. in $(0,T)\times M$. By Fubini's theorem, $v(t,\cdot)=\Psi(u(t,\cdot))$ $\mu$-a.e. for a.e. $t\in(0,T)$. Thus $u$ is a weak solution of \eqref{eq:1}.
	
	We now prove uniqueness. Let \(\hat u\) be another weak solution. Then, for every \(\varphi\in C_c^1([0,T);V)\),
	\begin{equation}
		-\int_0^T\int_M (u-\hat u)\partial_t\varphi\,d\mu\,dt
		+\int_0^T \mathcal{E}(\Psi(u(t))-\Psi(\hat u(t)),\varphi(t))\,dt = 0.
		\label{eq:uniq}
	\end{equation}
	By the same density argument as in Proposition~\ref{prop:ext}, \eqref{eq:uniq} holds for every \(\varphi\in\mathcal{H}\). Taking
	\[
	\varphi=\int_t^T (\Psi(u(s))-\Psi(\hat u(s)))\,ds
	\]
	and applying Lemma~\ref{lem:fubini}, we get
	\[
	\int_0^T\int_M (u-\hat u)(\Psi(u)-\Psi(\hat u))\,d\mu\,dt
	+\frac{1}{2}\mathcal{E}\left(\int_0^T (\Psi(u(s))-\Psi(\hat u(s)))\,ds\right)=0.
	\]
	Since the second term is nonnegative and \(\Psi\) is strictly monotone, we conclude \(u(t)=\hat u(t)\) for a.e. \(t\in(0,T)\).
	
	Corollary~\ref{cor;up} gives \(u\in L^\infty(0,T;L^{m+1})\cap W^{1,2}(0,T;V^*)\). Hence, it remains to show that $u$ satisfies \eqref{eq:bound} and \eqref{eq:boundv}. Indeed, noting that $u$ is the weak-* limits of $\{u_{n_j}\}_{j\geq 1}$ in $L^\infty(0,T;L^{m+1})$, we have by the weak-* lower semicontinuity of the norm in \(L^\infty(0,T;L^{m+1})\) and \eqref{ustNUB} that,
	\[
	\|u\|_{L^\infty(0,T;L^{m+1})}
	\le \liminf_{j\to\infty}\|u_{n_j}\|_{L^\infty(0,T;L^{m+1})}
	\le \|u_0\|_{L^{m+1}},
	\]
	thus proving \eqref{eq:bound}. Similarly, since $\Psi(u)=v$ almost everywhere and $v$ is the weak limit of $\{v_{n_j}\}_{j\geq 1}$ in \(L^2(0,T;V)\), \eqref{eq:boundv} can be obtained by using the weak lower semicontinuity of the norm in \(L^2(0,T;V)\) and \eqref{vNUB}.
\end{proof}

We now prove the comparison principle of \eqref{eq:cauchy}, i.e., Theorem~\ref{thm:main}(ii).


\begin{proof}[Proof of Theorem~\ref{thm:main}(ii)]
	Let \(\{\hat u_n\}_{n\ge1}\) and \(\{\tilde u_n\}_{n\ge1}\) be the approximate solutions constructed in Subsection~\ref{subsec:uvw} with initial data \(\hat u_0\) and \(\tilde u_0\), respectively. Since \(\hat u_0\le \tilde u_0\), Proposition~\ref{prop:elliptic-comparison} gives \(\hat u_n(t,\cdot)\le \tilde u_n(t,\cdot)\) for every \(t\in(0,T)\) and every \(n\ge1\).
	
	By the compactness argument in Section~\ref{Sec:Para}, there exists a common subsequence \(\{n_j\}_{j\ge1}\) such that
	\[
	\hat u_{n_j}\rightharpoonup \hat u,\qquad \tilde u_{n_j}\rightharpoonup \tilde u
	\quad\text{in } L^2(0,T;L^{m+1})
	\]
	as \(j\to\infty\). These limits are exactly the weak solutions with initial data \(\hat u_0\) and \(\tilde u_0\), respectively.
	
	Since \(\hat u_{n_j}\le \tilde u_{n_j}\) for every \(j\), for any nonnegative \(\varphi\in L^2(0,T;L^{1+1/m})\),
	\[
	\int_0^T\int_M (\hat u_{n_j}-\tilde u_{n_j})\varphi\,d\mu\,dt \le 0.
	\]
	Passing to the limit and using the weak convergences, we obtain
	\begin{equation}\label{eq:compare-test}
	\int_0^T\int_M (\hat u-\tilde u)\varphi\,d\mu\,dt \le 0
	\end{equation}
	for every nonnegative \(\varphi\in L^2(0,T;L^{1+1/m})\). 
	
	To conclude, we take $\varphi$ as a suitable approximation of the characteristic function of the set $\{\hat u>\tilde u\}$. More precisely, since $(M,d)$ is locally compact and separable, there exists a compact exhaustion $\{K_n\}_{n\geq 1}$---that is, for every $n\geq 1$, $K_n\subset M$ is compact and $K_n\subset K_{n+1}$, and $\cup_{n=1}^\infty K_n=M$. Hence, for every $n\geq 1$, define
	\[\varphi_n:=\mathbf{1}_{\{\hat u>\tilde u\}\cap\left([0,T]\times K_n\right)}.\]
	Since $\mu$ is a Radon measure $\varphi_n\in L^2(0,T;L^{1+1/m}))$ for every $n\geq 1$ and $\varphi_n\uparrow\mathbf{1}_{\{\hat u>\tilde u\}}$ as $n\to\infty$. Applying \eqref{eq:compare-test} with $\varphi_n$ for every $n\geq 1$ and then letting $n\to\infty$, we obtain by the monotone convergence theorem that
	\[
	\int_0^T\int_M (\hat u-\tilde u)_+\,d\mu\,dt \le 0.
	\]
	Hence $(\hat u-\tilde u)_+=0$ a.e. on $(0,T)\times M$, which implies $\hat u(t,\cdot)\le \tilde u(t,\cdot)$ $\mu$-a.e. for a.e. $t\in(0,T)$.

Finally, by Theorem~\ref{thm:main}(i), the zero function is the unique weak solution with zero initial datum. Applying the comparison principle above with $\hat u_0=0$ and $\tilde u_0=u_0$, we conclude that $u(t,\cdot)\ge0$ $\mu$-a.e. for almost all $t\in(0,T)$.
\end{proof}

\subsection{Stability with respect to initial data}
\label{subsec:stability}

In this subsection, we prove the continuous dependence of solutions on the initial data.

Let \(\{u_{0,k}\}_{k\ge1}\cup\{u_0\}\subset L^{m+1}(M,\mu)\) satisfy
\begin{equation}
	\lim_{k\to\infty}\|u_{0,k}-u_0\|_{L^{m+1}}=0. \label{eq:u0conv}
\end{equation}
Fix \(T>0\). Let \(u\) be the weak solution of \eqref{eq:cauchy} with initial datum \(u_0\), and for each \(k\ge1\), let \(u^{(k)}\) be the weak solution with initial datum \(u_{0,k}\). By Definition~\ref{def:weaksol}, for each \(k\ge1\),
\begin{equation}
	-\int_0^T\int_M (u^{(k)}-u)\varphi\,d\mu\,dt
	+\int_0^T \mathcal{E}(\Psi(u^{(k)})-\Psi(u),\varphi)\,dt
	= \int_M (u_{0,k}-u_0)\varphi(0)\,d\mu
	\label{eq:weakminus}
\end{equation}
for every \(\varphi\in C_c^1([0,T);V)\). By the density argument used in the proof of Proposition~\ref{prop:ext}, this identity extends to every \(\varphi\in\mathcal{H}\).

From Theorem~\ref{thm:main}(i) and Proposition~\ref{ParaUB}, we have the uniform bounds
\begin{equation}
	\|u^{(k)}\|_{L^\infty(0,T;L^{m+1})}\le \|u_{0,k}\|_{L^{m+1}},\qquad
	\|u\|_{L^\infty(0,T;L^{m+1})}\le \|u_0\|_{L^{m+1}},
	\label{eq:stab-bounds-linf}
\end{equation}
and
\begin{equation}
	\|\Psi(u^{(k)})\|_{L^2(0,T;V)}
	\le \sqrt{T\|u_{0,k}\|_{L^{m+1}}^{2m}+\frac{\|u_{0,k}\|_{L^{m+1}}^{m+1}}{m+1}},
	\qquad
	\|\Psi(u)\|_{L^2(0,T;V)}
	\le \sqrt{T\|u_0\|_{L^{m+1}}^{2m}+\frac{\|u_0\|_{L^{m+1}}^{m+1}}{m+1}}.
	\label{eq:stab-bounds-v}
\end{equation}
Combining these with \eqref{eq:u0conv}, we obtain uniform bounds: there exist constants \(C_0>0\), independent of \(T\), and \(C_T>0\), depending on \(T\), such that
\begin{align}
	\sup_{k\ge1}\|u^{(k)}\|_{L^\infty(0,T;L^{m+1})}&\le C_0, \qquad
	\|u\|_{L^\infty(0,T;L^{m+1})}\le C_0, \label{eq:hatu-bound}\\
	\sup_{k\ge1}\|\Psi(u^{(k)})\|_{L^2(0,T;V)}&\le C_T, \qquad
	\|\Psi(u)\|_{L^2(0,T;V)}\le C_T. \label{eq:hatu-boundv}
\end{align}

\begin{proposition}\label{prop:stab-weak}
	The sequence \(\{u^{(k)}\}_{k\ge1}\) converges weakly to \(u\) in \(L^2(0,T;L^{m+1})\).
\end{proposition}

\begin{proof}
	From \eqref{eq:hatu-bound} and H\"{o}lder's inequality, \(\{u^{(k)}\}_{k\ge1}\) is bounded in \(L^2(0,T;L^{m+1})\); from \eqref{eq:hatu-boundv}, \(\{\Psi(u^{(k)})\}_{k\ge1}\) is bounded in \(L^2(0,T;V)\). Hence there exist a subsequence \(\{k_j\}_{j\ge1}\subset\mathbb{N}\) with \(k_j\uparrow\infty\), and functions
	\[
	\tilde u\in L^2(0,T;L^{m+1}),\qquad \tilde v\in L^2(0,T;V),
	\]
	such that
	\begin{align}
		u^{(k_j)} &\rightharpoonup \tilde u &&\text{in } L^2(0,T;L^{m+1}), \label{eq:hatu-weak} \\
		\Psi(u^{(k_j)}) &\rightharpoonup \tilde v &&\text{in } L^2(0,T;V). \label{eq:hatv-weak}
	\end{align}
	Letting \(j\to\infty\) in \eqref{eq:weakminus} with \(k=k_j\), and using \eqref{eq:hatu-weak}, \eqref{eq:hatv-weak}, and \eqref{eq:u0conv}, we obtain
	\[
	-\int_0^T\int_M (\tilde u-u)\varphi\,d\mu\,dt
	+\int_0^T \mathcal{E}(\tilde v-\Psi(u),\varphi)\,dt=0
	\]
	for every \(\varphi\in\mathcal{H}\). Taking
	\[
	\varphi(t):=\int_t^T (\tilde v(s)-\Psi(u(s)))\,ds
	\]
	yields
	\begin{equation}
		\int_0^T\int_M (\tilde u-u)(\tilde v-\Psi(u))\,d\mu\,dt
		+\frac{1}{2}\mathcal{E}\left(\int_0^T (\tilde v(s)-\Psi(u(s)))\,ds\right)=0.
		\label{eq:stab1}
	\end{equation}
	
	On the other hand, for each \(k\ge1\), taking
	\[
	\varphi(t):=\int_t^T (\Psi(u^{(k)}(s))-\Psi(u(s)))\,ds
	\]
	in \eqref{eq:weakminus}, we get
	\begin{equation}
		\begin{aligned}
			&\int_0^T\int_M (u^{(k)}-u)(\Psi(u^{(k)})-\Psi(u))\,d\mu\,dt
			+\frac{1}{2}\mathcal{E}\left(\int_0^T (\Psi(u^{(k)}(s))-\Psi(u(s)))\,ds\right) \\
			&= \int_M (u_{0,k}-u_0)\int_0^T (\Psi(u^{(k)}(s))-\Psi(u(s)))\,ds\,d\mu.
		\end{aligned}
		\label{eq:stab2}
	\end{equation}
	From \eqref{eq:hatu-bound},
	\[
	\left\|\int_0^T (\Psi(u^{(k)}(s))-\Psi(u(s)))\,ds\right\|_{L^{1+1/m}}
	\le \int_0^T \bigl(\|u^{(k)}(s)\|_{L^{m+1}}^m+\|u(s)\|_{L^{m+1}}^m\bigr)\,ds
	\le 2TC_0^m.
	\]
	Together with \eqref{eq:u0conv}, this gives
	\begin{equation}
		\lim_{k\to\infty}
		\int_M (u_{0,k}-u_0)\int_0^T (\Psi(u^{(k)}(s))-\Psi(u(s)))\,ds\,d\mu = 0.
		\label{eq:m+1}
	\end{equation}
	Letting \(j\to\infty\) in \eqref{eq:stab2} with \(k=k_j\), and using \eqref{eq:Eliminf} and \eqref{eq:m+1}, we obtain
	\[
	\limsup_{j\to\infty}
	\int_0^T\int_M (u^{(k_j)}-u)(\Psi(u^{(k_j)})-\Psi(u))\,d\mu\,dt
	\le -\frac{1}{2}\mathcal{E}\left(\int_0^T (\tilde v(s)-\Psi(u(s)))\,ds\right).
	\]
	Comparing this with \eqref{eq:stab1}, we have
	\[
	\limsup_{j\to\infty}
	\int_0^T\int_M (u^{(k_j)}-u)(\Psi(u^{(k_j)})-\Psi(u))\,d\mu\,dt
	\le \int_0^T\int_M (\tilde u-u)(\tilde v-\Psi(u))\,d\mu\,dt.
	\]
	By \eqref{eq:hatu-weak} and \eqref{eq:hatv-weak},
	\[
	\limsup_{j\to\infty}
	\int_0^T\int_M u^{(k_j)}\Psi(u^{(k_j)})\,d\mu\,dt
	\le \int_0^T\int_M \tilde u\,\tilde v\,d\mu\,dt.
	\]
	By the Minty trick as in the proof of Theorem~\ref{thm:main}(i)(cf. \S~\ref{sub:unicom}), it follows that
	\[
	\tilde v(t,\cdot)=\Psi(\tilde u(t,\cdot)) \quad \mu\text{-a.e. for a.e. }t\in(0,T).
	\]
	Thus \(\tilde u\) is also a weak solution of \eqref{eq:cauchy} with initial datum \(u_0\). By the uniqueness part in Theorem~\ref{thm:main}(i), \(\tilde u=u\) a.e. on \((0,T)\times M\).
	
	It remains to show that the whole sequence \(\{u^{(k)}\}_{k\ge1}\) converges weakly to \(u\). Otherwise, there would exist a functional \(h\in L^2(0,T;L^{m+1})^*\), a constant \(\varepsilon_0>0\), and a subsequence \(\{u^{(\ell_j)}\}_{j\ge1}\subset\{u^{(k)}\}_{k\ge1}\) such that
	\begin{equation}
		\left\langle h, u^{(\ell_j)}-u\right\rangle_{L^2(0,T;L^{m+1})}\ge \varepsilon_0
		\label{eq:fzf}
	\end{equation}
	for every \(j\ge1\). But \(\{u^{(\ell_j)}\}_{j\ge1}\) is also bounded in \(L^2(0,T;L^{m+1})\), hence admits a weakly convergent subsequence; by the argument above, its weak limit must be \(u\), contradicting \eqref{eq:fzf}. Therefore \(u^{(k)}\rightharpoonup u\) in \(L^2(0,T;L^{m+1})\).
\end{proof}

\begin{proposition}\label{prop:stab-norm}
	We have
	\begin{equation}
		\lim_{k\to\infty}\|u^{(k)}\|_{L^2(0,T;L^{m+1})}=\|u\|_{L^2(0,T;L^{m+1})}.
		\label{eq:stab-norm}
	\end{equation}
\end{proposition}

\begin{proof}
	From \eqref{eq:stab2},
	\begin{equation}
		\int_0^T\int_M (u^{(k)}-u)(\Psi(u^{(k)})-\Psi(u))\,d\mu\,dt
		\le \int_M (u_{0,k}-u_0)\int_0^T (\Psi(u^{(k)}(s))-\Psi(u(s)))\,ds\,d\mu.
		\label{eq:stab3}
	\end{equation}
	For the left-hand side, by H\"older's inequality,
	\begin{align}
		&\int_0^T\int_M (u^{(k)}-u)(\Psi(u^{(k)})-\Psi(u))\,d\mu\,dt \notag \\
		&= \int_0^T\int_M \left(|u^{(k)}|^{m+1}+|u|^{m+1}-u^{(k)}\Psi(u)-u\Psi(u^{(k)})\right)\,d\mu\,dt \notag \\
		&\ge \int_0^T \left(\|u^{(k)}(t)\|_{L^{m+1}}^{m+1}+\|u(t)\|_{L^{m+1}}^{m+1}\right)dt \notag \\
		&\qquad -\int_0^T \left(\|u^{(k)}(t)\|_{L^{m+1}}\|u(t)\|_{L^{m+1}}^m+\|u(t)\|_{L^{m+1}}\|u^{(k)}(t)\|_{L^{m+1}}^m\right)dt \notag \\
		&= \int_0^T \left(\|u^{(k)}(t)\|_{L^{m+1}}-\|u(t)\|_{L^{m+1}}\right)
		\left(\|u^{(k)}(t)\|_{L^{m+1}}^m-\|u(t)\|_{L^{m+1}}^m\right)dt \ge 0.
		\label{eq:stab4}
	\end{align}
	For every $k\geq 1$, set
	\[
	a_k(t):=\|u^{(k)}(t)\|_{L^{m+1}},\qquad a(t):=\|u(t)\|_{L^{m+1}}.
	\]
	Then \eqref{eq:stab4} becomes
	\begin{equation}
		\int_0^T\int_M (u^{(k)}-u)(\Psi(u^{(k)})-\Psi(u))\,d\mu\,dt
		\ge \int_0^T (a_k(t)-a(t))(a_k(t)^m-a(t)^m)\,dt \ge 0.
		\label{eq:stab5}
	\end{equation}
	Letting \(k\to\infty\) in \eqref{eq:stab3} and using \eqref{eq:stab5} and \eqref{eq:m+1}, we obtain
	\begin{equation}
		\lim_{k\to\infty}
		\int_0^T (a_k(t)-a(t))\left(a_k(t)^m-a(t)^m\right)\,dt = 0.
		\label{eq:stab6}
	\end{equation}
	For any \(\varepsilon>0\), define
	\[
	A_k(\varepsilon):=\{t\in[0,T): |a_k(t)-a(t)|\ge\varepsilon\}.
	\]
	From \eqref{eq:hatu-bound}, it follows that
	\begin{equation}\label{eq:a-bound}
	\sup_{k\ge1}\esup_{t\in[0,T)} a_k(t)\le C_0,\qquad
	\esup_{t\in[0,T)} a(t)\le C_0.
	\end{equation}
	Note by basic calculus that, for any $varepsilon>0$, the minimum of function $F(s,t)=(s-t)(s^m-t^m)$ in $[0,C_0]\times[0,C_0]\setminus\{|s-t|\geq\varepsilon\}$ is a strictly positive real number depending only on $\varepsilon,m,C_0$. Denote by $\delta:=\delta(\varepsilon,m,C_0)$ this minimum. Then, 
	\[
	0\le \delta|A_k(\varepsilon)|
	\le \int_0^T (a_k(t)-a(t))(a_k(t)^m-a(t)^m)\,dt.
	\]
	Combining this inequality with \eqref{eq:stab6}, we obtain
	\begin{equation}
		\lim_{k\to\infty}|A_k(\varepsilon)|=0
		\label{eq:stab7}
	\end{equation}
	for every \(\varepsilon>0\), i.e. \(a_k\to a\) in measure.
	
	Therefore, for every \(\varepsilon>0\) and \(k\ge1\), it follows from \eqref{eq:hatu-bound} that
	\begin{equation}\label{eq:last5}
	\left|\|u^{(k)}\|_{L^2(0,T;L^{m+1})}^2-\|u\|_{L^2(0,T;L^{m+1})}^2\right|
	= \left|\int_0^T (a_k(t)^2-a(t)^2)\,dt\right|
	\le 2C_0\int_0^T |a_k(t)-a(t)|\,dt.
	\end{equation}
	Using the decomposition \([0,T)=A_k(\varepsilon)\cup([0,T)\setminus A_k(\varepsilon))\), we get
	\[
	\int_0^T |a_k(t)-a(t)|\,dt
	\le 2C_0|A_k(\varepsilon)| + T\varepsilon.
	\]
	Letting \(k\to\infty\) in this inequality, we have
	\[\limsup_{k\to\infty} \int_0^T |a_k(t)-a(t)|\,dt\leq T\varepsilon.\]
	Since \(\varepsilon\downarrow0\) is arbitrary, letting $\varepsilon\downarrow0^+$ and using \eqref{eq:last5}, we obtain
	\[\lim_{k\to\infty}\left|\|u^{(k)}\|_{L^2(0,T;L^{m+1})}^2-\|u\|_{L^2(0,T;L^{m+1})}^2\right|=0,\]
	showing \eqref{eq:stab-norm}.
\end{proof}

We now prove Theorem~\ref{thm:main}(iii).

\begin{proof}[Proof of Theorem~\ref{thm:main}(iii)]
	By Proposition~\ref{prop:stab-weak}, \(u^{(k)}\rightharpoonup u\) in \(L^2(0,T;L^{m+1})\). By Proposition~\ref{prop:stab-norm}, the norms converge. Since \(L^2(0,T;L^{m+1})\) is a uniformly convex Banach space, weak convergence together with convergence of norms implies strong convergence. Hence
	\[
	\lim_{k\to\infty}\|u^{(k)}-u\|_{L^2(0,T;L^{m+1})}=0,
	\]
	which proves Theorem~\ref{thm:main}(iii).
\end{proof}

\section{Applications to specific metric measure spaces}
\label{Sec:Outro}

In this section, we illustrate the main conclusions of Theorem~\ref{thm:main}—existence, uniqueness, and nonnegativity preservation—to concrete metric measure spaces equipped with Dirichlet forms, thereby obtaining well-posedness results for the porous medium equation in a variety of settings. The examples we consider range from classical Euclidean spaces and smooth Riemannian manifolds to non-smooth fractal spaces and general metric measure spaces. In the Euclidean and Riemannian cases, our results are compatible with known ones. In the fractal and general metric measure settings, the results appear to be new.

\subsection{Euclidean spaces}
\label{exmp:mani}

We first illustrate Theorem~\ref{thm:main} in the classical Euclidean setting.

\subsubsection{Signed PME and signed FDE on Euclidean spaces}

Let \((M,d)=(\mathbb{R}^N,|\cdot|)\) with \(N\ge2\), \(\mu=dx\) the Lebesgue measure on \(\mathbb{R}^N\), and let \(\mathcal{L}\) be the Laplace operator on \(\mathbb{R}^N\). Then equation \eqref{eq:1} reduces to the porous medium equation (\(m>1\)) or the fast diffusion equation (\(0<m<1\)).

Recall from Example~\ref{ex:dirichlet}(i) that \(\mathcal{L}\) is associated with the Dirichlet form
\[
\left(\int_{\mathbb{R}^N}\nabla u\cdot\nabla v\,dx,\; W^{1,2}(\mathbb{R}^N)\right).
\]
By \cite[Examples 1.5.3 \& 1.6.2]{FukushimaOshimaTakeda.2011.489}, the extended Dirichlet space satisfies
\[
W_e^{1,2}(\mathbb{R}^N)=
\begin{cases}
	\dot{W}^{1,2}(\mathbb{R}^N)/\mathbb{R}, & N\ge3,\\
	\dot{W}^{1,2}(\mathbb{R}^N), & N=2,
\end{cases}
\]
where
\[
\dot{W}^{1,2}(\mathbb{R}^N):=\{u\in L_{\mathrm{loc}}^2(\mathbb{R}^N):\nabla u\in L^2(\mathbb{R}^N)\}
\]
is the homogeneous Sobolev space.

Applying Theorem~\ref{thm:main}(i),(ii) yields the following consequence.

\begin{corollary}
	\label{cor:euclid}
	Set \(V:=L^{1+1/m}\cap \dot{W}^{1,2}\). For every \(u_0\in L^{m+1}(\mathbb{R}^N,dx)\), there exists a unique
	\[
	u\in L^\infty(0,T;L^{m+1}(\mathbb{R}^N,dx))\cap W^{1,2}(0,T;V^*)
	\]
	such that \(\Psi(u)\in L^2(0,T;V)\) and
	\[
	\int_0^T\int_{\mathbb{R}^N}-u\,\partial_t\varphi\,dx\,dt
	+\int_0^T\int_{\mathbb{R}^N} \nabla u^m\cdot\nabla\varphi\,dx\,dt
	= \int_{\mathbb{R}^N} u_0\varphi(0,\cdot)\,dx
	\]
	for every \(\varphi\in C_c^1([0,T);V)\). Moreover, if \(u_0\ge0\) a.e., then \(u(t,\cdot)\ge0\) a.e. for a.e. \(t\in(0,T)\).
\end{corollary}

This is compatible with \cite[Definition 9.3 \& Theorem 9.25]{VazquezPMEBook} and slightly extends the class of admissible initial data.

\subsubsection{Signed fractional PME and signed fractional FDE on Euclidean spaces}

Let \((M,d,\mu)\) be as in the previous subsection, but now let \(\mathcal{L}=-(-\Delta)^{\beta/2}\) be the fractional Laplace operator on \(\mathbb{R}^N\) for some \(0<\beta<2\). Then equation \eqref{eq:1} reduces to the fractional porous medium equation (\(m>1\)) or the corresponding fast diffusion equation (\(0<m<1\)).

Recall from Example~\ref{ex:dirichlet}(ii) that \(-(-\Delta)^{\beta/2}\) is associated with the Dirichlet form \((\mathcal{E}_\beta,\mathcal{F}_\beta)\), where \(\mathcal{F}_\beta=W^{\beta/2,2}(\mathbb{R}^N)\). Since \(\beta<2\le N\), by \cite[Examples 1.5.2, (1.5.17)--(1.5.19)]{FukushimaOshimaTakeda.2011.489}, the extended Dirichlet space satisfies
\[
W_e^{\beta/2,2}(\mathbb{R}^N)=\dot{L}_{\beta/2,2}(\mathbb{R}^N),
\]
the Riesz potential space of index \(\beta/2\).

Applying Theorem~\ref{thm:main}(i),(ii) yields the following consequence.

\begin{corollary}
	\label{cor:fractional}
	Set \(V:=L^{1+1/m}(\mathbb{R}^N,dx)\cap \dot{L}_{\beta/2,2}\). For every \(u_0\in L^{m+1}(\mathbb{R}^N)\), there exists a unique
	\[
	u\in L^\infty(0,T;L^{m+1}(\mathbb{R}^N,dx))\cap W^{1,2}(0,T;V^*)
	\]
	such that \(\Psi(u)\in L^2(0,T;V)\) and
	\[
	\int_0^T\int_{\mathbb{R}^N} -u\,\partial_t\varphi\,dx\,dt
	+\int_0^T\int_{\mathbb{R}^N} |\xi|^\beta \widehat{u}(\xi)\overline{\widehat{\varphi}(\xi)}\,d\xi\,dt
	= \int_{\mathbb{R}^N} u_0\varphi(0,\cdot)\,dx
	\]
	for every \(\varphi\in C_c^1([0,T);V)\). Moreover, if \(u_0\ge0\) a.e., then \(u(t,\cdot)\ge0\) a.e. for a.e. \(t\in(0,T)\).
\end{corollary}

This is compatible with \cite[Definition 3.1]{PabloQRVazquez2012CPAM} and slightly extends the class of admissible initial data.

\subsection{Riemannian manifolds}

Let \((M,g)\) be a geodesically complete Riemannian manifold with volume measure \(\mu\) and Laplace--Beltrami operator \(\Delta_g\). As recalled in Example~\ref{ex:dirichlet}(iii), the quadratic form
\[
\mathcal{E}(u,v)=\int_M (\nabla u,\nabla v)_g\,d\mu
\]
is closable in \(L^2(M,\mu)\), and its closure is a regular local Dirichlet form on \(L^2(M,\mu)\) with domain \(W^{1,2}(M)\); its associated Laplacian is the associated self-adjoint operator \(\Delta_g\). Applying Theorem~\ref{thm:main}(i),(ii) yields the same well-posedness result as in the Euclidean case (cf. Corollary~\ref{cor:euclid}).

\subsection{Sierpi\'nski gasket and infinite Sierpi\'nski gasket}

\label{exmp:frac}

We now apply Theorem~\ref{thm:main} to the Sierpi\'nski gasket, considering both the compact case (with Cauchy--Dirichlet problem) and the non-compact blow-up (with Cauchy problem).

\subsubsection{The Cauchy--Dirichlet problem on the Sierpi\'nski gasket}

Let $\{q_1,q_2,q_3\}$ be the vertices of the unit equilateral triangle in $\mathbb{R}^2$, say
\[
\qquad q_1=(0,0),\qquad q_2=(1,0),\qquad q_3=\left(\frac12,\frac{\sqrt3}{2}\right).
\]
For $i=1,2,3$, define the contraction maps
\[
F_i(x):=\frac{x+q_i}{2},\qquad x\in\mathbb{R}^2.
\]
The Sierpi\'nski gasket $K$ is the unique non-empty compact set satisfying
\[
K=\bigcup_{i=1}^3 F_i(K),
\]
as shown in Figure \ref{fig:sierpinski}.
\begin{figure}[htbp]
	\centering
	\includegraphics[width=0.4\textwidth]{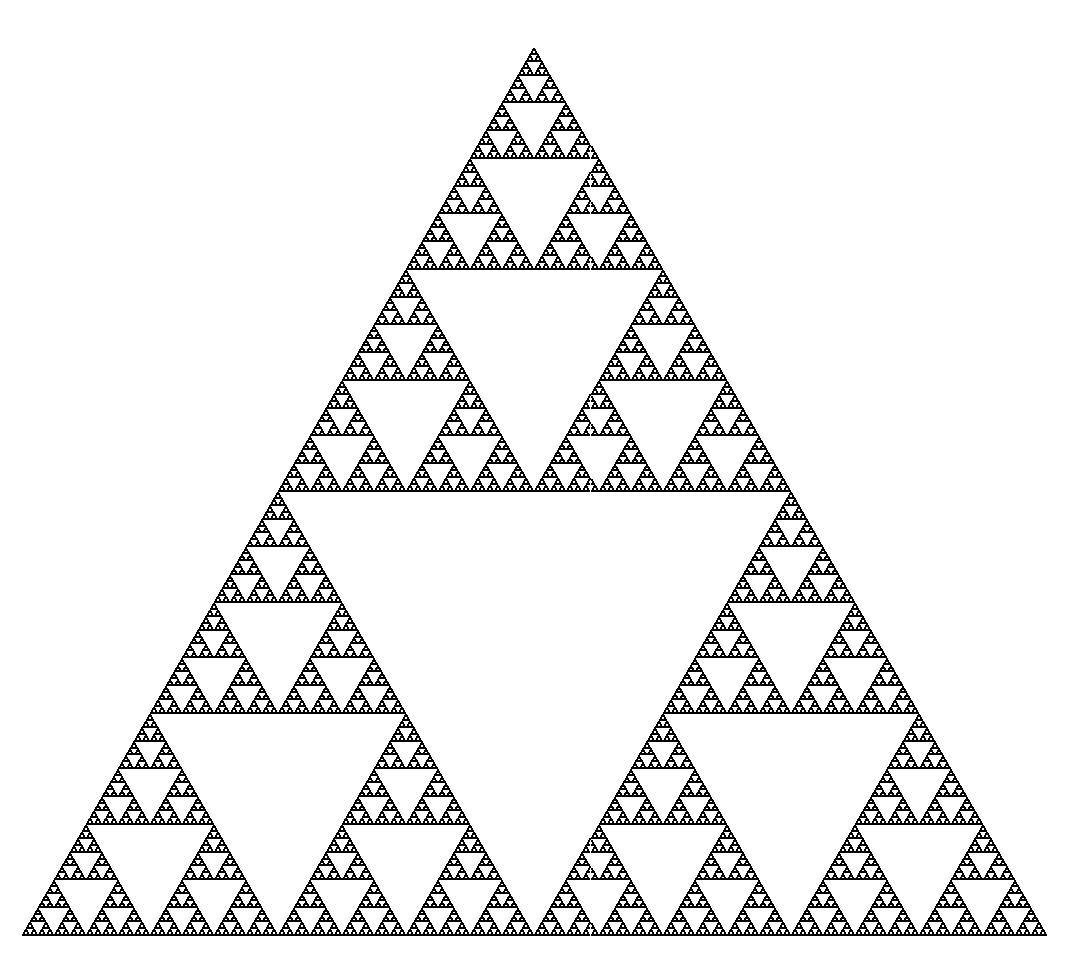}
	\caption{The Sierpi\'nski gasket \(K\) .}
	\label{fig:sierpinski}
\end{figure}

Let $\mu$ be any fixed self-similar measure on $K$, that is, a Borel probability measure with full support satisfying
\[
\mu(A)=\sum_{i=1}^3 \mu_i\,\mu(F_i^{-1}(A))
\]
for some weights $\mu_1,\mu_2,\mu_3\in(0,1)$ with $\mu_1+\mu_2+\mu_3=1$. Then $(K,|\cdot|,\mu)$ is a compact metric measure space in the sense of Section~\ref{sec:Intro}.

For each $n\ge0$, define the boundary set $V_0:=\{q_1,q_2,q_3\}$ and the $n$-th level vertex set
\[
V_n:=\bigcup_{\tau\in I^n} F_\tau(V_0),
\]
where $I:=\{1,2,3\}$, $I^0:=\{\varnothing\}$, and $F_\tau:=F_{\tau_1}\circ\cdots\circ F_{\tau_n}$ for $\tau=\tau_1\cdots\tau_n\in I^n$. The sets $\{V_n\}_{n\ge0}$ form an increasing sequence of finite subsets of $K$, and $\overline{\bigcup_{n\ge0}V_n}=K$.

We now recall the construction of the Dirichlet form on $K$ due to Kigami \cite{Kigami.1993.TAMS721}. On each finite set $V_n$, define the discrete energy
\[
\mathcal{E}_n(u,u):=\left(\frac53\right)^n
\sum_{\tau\in I^n}\sum_{i<j}\bigl(u(F_\tau(q_i))-u(F_\tau(q_j))\bigr)^2,
\qquad u\in\ell(V_n).
\]
The renormalization factor $5/3$ ensures the compatibility condition $\mathcal{E}_{n+1}|_{V_n}=\mathcal{E}_n$ after harmonic extension. For any $u\in C(K)$, the sequence $\{\mathcal{E}_n(u|_{V_n},u|_{V_n})\}_{n\ge1}$ is non-decreasing in $n$. Hence, define
\[
\mathcal{F}:=\left\{u\in C(K):\sup_{n\ge0}\mathcal{E}_n\left(u|_{V_n},u|_{V_n}\right)<\infty\right\},
\]
and
\[
\mathcal{E}(u,u):=\lim_{n\to\infty}\mathcal{E}_n(u|_{V_n},u|_{V_n}),\qquad u\in\mathcal{F}.
\]
The pair $(\mathcal{E},\mathcal{F})$ is a strongly local, regular, self-similar Dirichlet form on $L^2(K,\mu)$ (see \cite[Section 9]{Kigami.2012.MAMS132}), and satisfies the following Sobolev-type inequality (see \cite{FukushimaShima.1992.PA1} or \cite[Lemma 2.4]{FalconerHu.1999.JMAA552}): there exists a constant $C>0$ such that, for any $u\in\mathcal{F}$ and any $x,y\in K$,
\begin{equation} \label{eq:soboinq1}
	|u(x)-u(y)|^2 \le C|x-y|^\gamma \mathcal{E}(u),
\end{equation}
where $\gamma:=\log_2(5/3)>0$.

To define the Cauchy--Dirichlet problem with zero boundary values, we introduce the space $\mathcal{F}_0$ as the closure of $\mathcal{F}\cap C_c(K\setminus V_0)$ in $\mathcal{F}$ with respect to the $\mathcal{E}_1$-norm. Then $(\mathcal{E},\mathcal{F}_0)$ is again a regular Dirichlet form on $L^2(K,\mu)$, and its associated Laplacian will be denoted by $\mathcal{L}_{\mu,D}$. By \eqref{eq:soboinq1}, every function in $\mathcal{F}_0$ vanishes on $V_0$, so the zero boundary condition is encoded in the domain.

We consider the following Cauchy--Dirichlet problem:
\begin{equation} \label{eq:CD-problem}
	\begin{cases}
		\partial_t u = \mathcal{L}_{\mu,D}\left(|u|^{m-1}u\right), & (t,x)\in(0,T)\times K,\\
		u(0,x)=u_0(x), & x\in K,
	\end{cases}
\end{equation}
where the boundary condition is understood as $u(t,\cdot)\in\mathcal{F}_0$ for a.e. $t\in(0,T)$, i.e., the functions $u(t,\cdot)$ vanish on $V_0$.

The following observation is crucial for applying Theorem~\ref{thm:main}.

\begin{proposition}
	For the Dirichlet form $(\mathcal{E},\mathcal{F}_0)$ on the Sierpi\'nski gasket,
	\[
	\mathcal{F}_{0,e}=\mathcal{F}_0,
	\]
	where $\mathcal{F}_{0,e}$ denotes the extended Dirichlet space of $(\mathcal{E},\mathcal{F}_0)$.
\end{proposition}

\begin{proof}
	It suffices to prove $\mathcal{F}_{0,e}\subset\mathcal{F}_0$.
	
	For any $u\in\mathcal{F}_0$, by definition there exists a sequence $\{\phi_n\}\subset\mathcal{F}\cap C_c(K\setminus V_0)$ converging to $u$ in the $\mathcal{E}_1$-norm. Since each $\phi_n$ vanishes on $V_0$, applying \eqref{eq:soboinq1} to $\phi_n$ with $y=q_1\in V_0$ yields
	\[
	|\phi_n(x)|^2 \le C|x-q_1|^\gamma \mathcal{E}(\phi_n) \le C\,\mathcal{E}(\phi_n),
	\]
	where the last inequality uses $\operatorname{diam}(K)=1$. Integrating over $K$ gives
	\[
	\|\phi_n\|_{L^2(K,\mu)}^2 \le C\,\mathcal{E}(\phi_n).
	\]
	Passing to the limit as $n\to\infty$, we obtain
	\begin{equation} \label{eq:soboinq2}
		\|u\|_{L^2(K,\mu)}^2 \le C\,\mathcal{E}(u),\qquad u\in\mathcal{F}_0.
	\end{equation}
	
	Now fix $v\in\mathcal{F}_{0,e}$ and let $\{v_n\}\subset\mathcal{F}_0$ be an approximating sequence for $v$ in the extended sense: $v_n\to v$ $\mu$-a.e. and $\mathcal{E}(v_n-v)\to0$. Applying \eqref{eq:soboinq2} to $v_n$ and passing to the limit via Fatou's lemma yields
	\[
	\|v\|_{L^2(K,\mu)}^2
	\le \liminf_{n\to\infty}\|v_n\|_{L^2(K,\mu)}^2
	\le C\lim_{n\to\infty}\mathcal{E}(v_n)
	= C\,\mathcal{E}(v),
	\]
	showing that $v\in\mathcal{F}_0$. Thus $\mathcal{F}_{0,e}\subset\mathcal{F}_0$. The reverse inclusion is immediate.
\end{proof}

With the identification $\mathcal{F}_{0,e}=\mathcal{F}_0$, the Cauchy--Dirichlet problem \eqref{eq:CD-problem} fits exactly into the framework of Theorem~\ref{thm:main} with $V:=\mathcal{F}_0$. Note that $\mathcal{F}_0\subset C(K)$, $K$ is compact and $\mu(V_0)=0$, so $\mathcal{F}_0\subset L^{1+1/m}(K,\mu)=L^{1+1/m}(K\setminus V_0,\mu)$. Therefore, we obtain the following.

\begin{corollary}
	\label{cor:fra}
	For every $T>0$ and every $u_0\in L^{m+1}(K,\mu)$, there exists a unique
	\[
	u\in L^\infty(0,T;L^{m+1}(K,\mu))\cap W^{1,2}(0,T;\mathcal{F}_0^*)
	\]
	such that
	\[
	\Psi(u):=|u|^{m-1}u\in L^2(0,T;\mathcal{F}_0),
	\]
	and $u$ satisfies
	\[
	-\int_0^T\int_K u\,\partial_t\varphi\,d\mu\,dt
	+\int_0^T \mathcal{E}\bigl(\Psi(u)(t),\varphi(t)\bigr)\,dt
	= \int_K u_0\,\varphi(0,\cdot)\,d\mu
	\]
	for every $\varphi\in C_c^1([0,T);\mathcal{F}_0)$. Moreover, if $u_0\ge0$ $\mu$-a.e., then $u(t,\cdot)\ge0$ $\mu$-a.e. for a.e. $t\in(0,T)$.
\end{corollary}

\begin{remark}
	The proof of Corollary~\ref{cor:fra} relies essentially only on two structural properties of the Dirichlet form $(\mathcal{E},\mathcal{F})$: the embedding $\mathcal{F}\hookrightarrow C(K)$ and the Sobolev-type inequality
	\[
	|u(x)-u(y)|^2 \le C|x-y|^\gamma \mathcal{E}(u),\qquad u\in\mathcal{F},
	\]
	for some constants $C>0$ and $\gamma>0$ depending on the fractal. These properties hold for a broad class of fractals, including all nested fractals \cite{Lindstrom.1990.MAMS128} and the Sierpi\'nski carpet \cite{KusuokaYin.1992.PTRF169}. Consequently, the same well-posedness result as in Corollary~\ref{cor:fra} holds for the Cauchy--Dirichlet problem on these spaces, with $\mathcal{F}_0$ replaced by the corresponding Dirichlet subspace (and $V_0$ understood as the natural boundary of the fractal; see, e.g., \cite{CaoQiu2022JFA}).
\end{remark}

\subsubsection{The Cauchy problem on the infinite Sierpi\'nski gasket}

We now consider the unbounded version of the Sierpi\'nski gasket, obtained by a blow-up construction (see, e.g., \cite{Sabot2000JFA} or \cite{Teplyaev1998JFA}). Let the contraction maps $F_1,F_2,F_3$ and the compact Sierpi\'nski gasket $K$ be as above, with the standard Dirichlet form $(\mathcal{E},\mathcal{F})$ on $L^2(K,\nu)$, where $\nu$ is the self-similar measure with equal weights $1/3,1/3,1/3$.

Fix an infinite backward sequence $S:=\{i_{-n}\}_{n\ge1}\subset I$. For each $n\ge1$, define
\[
F_{-n,S}:=F_{i_{-1}}^{-1}\circ F_{i_{-2}}^{-1}\circ\cdots\circ F_{i_{-n}}^{-1},
\qquad K_{-n}:=F_{-n,S}(K),
\]
as illustrated in Figure~\ref{fig:unbounded_sg}.
\begin{figure}[!htbp]
	\centering
	\includegraphics[width=0.4\textwidth]{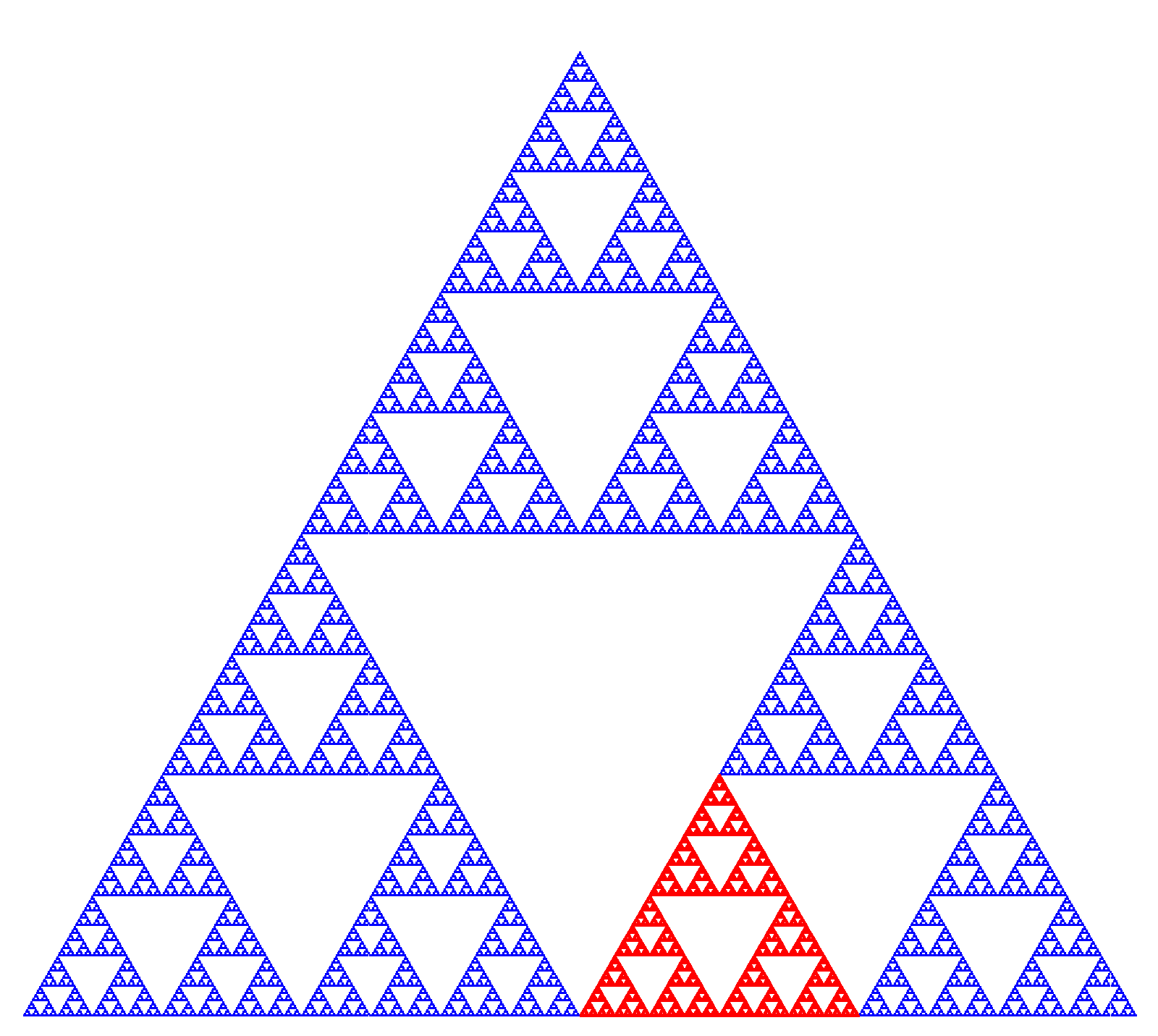}
	\caption{Example of the two-level blow-up $K_{-2}$ for the infinite Sierpiński gasket with reverse sequence $(i_{-1},i_{-2})=(2,1)$. The red subset indicates the original Sierpiński gasket $K$, which is contained in $K_{-2}$.}
	\label{fig:unbounded_sg}
\end{figure}

Since each $K_{-n}$ is a compact subset, the sets $\{K_{-n}\}_{n\ge1}$ form an increasing sequence and provide a compact exhaustion of the unbounded set
\[
K_{-\infty}:=\bigcup_{n\ge1}K_{-n}.
\]

The space $K_{-\infty}$ is non-compact. To equip it with a measure, for each $n\ge1$, define $\nu_{-n}$ on $K_{-n}$ by
\[
\int_{K_{-n}} f\,d\nu_{-n}
= \sum_{\tau\in I^n}\int_{F_{-n,S}\circ F_\tau(K)} f\,d\nu
= 3^n\int_K f\circ F_{-n,S}\,d\nu,
\qquad f\in C(K_{-n}).
\]
The measures $\nu_{-n}$ are compatible with the inclusions $K_{-n}\subset K_{-(n+1)}$, and hence extend to a Radon measure $\nu_{-\infty}$ on $K_{-\infty}$ with $\nu_{-\infty}(K_{-\infty})=\infty$.

The Dirichlet form on $K_{-\infty}$ is obtained by scaling the compact one. For each $n\ge1$, define
\[
\mathcal{F}_{-n}:=\{g\circ F_{-n,S}:g\in\mathcal{F}\},
\qquad
\mathcal{E}_{-n}(u,u):=\left(\frac35\right)^n \mathcal{E}(u\circ F_{-n,S},u\circ F_{-n,S}),\quad u\in\mathcal{F}_{-n}.
\]
Let
\[
\mathcal{F}_{-\infty}:=\left\{
u\in L^2(K_{-\infty},\nu_{-\infty}):
u|_{K_{-n}}\in\mathcal{F}_{-n}\ \forall n\ge1,\
\sup_{n\ge1}\mathcal{E}_{-n}\left(u|_{K_{-n}},u|_{K_{-n}}\right)<\infty
\right\},
\]
and for $u\in\mathcal{F}_{-\infty}$, define
\[
\mathcal{E}_{-\infty}(u,u):=\lim_{n\to\infty}\mathcal{E}_{-n}\left(u|_{K_{-n}},u|_{K_{-n}}\right).
\]
By \cite[Theorem 6.4]{Fukushima.1992.151}, the pair $(\mathcal{E}_{-\infty},\mathcal{F}_{-\infty})$ is a local, regular Dirichlet form on $L^2(K_{-\infty},\nu_{-\infty})$; its associated Laplacian is denoted by $\mathcal{L}_{-\infty}$.

Applying Theorem~\ref{thm:main}(i),(ii) with $V:=L^{1+1/m}(K_{-\infty},\nu_{-\infty})\cap\mathcal{F}_{-\infty,e}$ yields the following.

\begin{corollary}
	Set $V:=L^{1+1/m}(K_{-\infty},\nu_{-\infty})\cap\mathcal{F}_{-\infty,e}$. For every $T>0$ and every $u_0\in L^{m+1}(K_{-\infty},\nu_{-\infty})$, there exists a unique
	\[
	u\in L^\infty(0,T;L^{m+1}(K_{-\infty},\nu_{-\infty}))\cap W^{1,2}(0,T;V^*)
	\]
	such that
	\[
	\Psi(u):=|u|^{m-1}u\in L^2(0,T;V),
	\]
	and $u$ satisfies the weak formulation
	\[
	-\int_0^T\int_{K_{-\infty}} u\,\partial_t\varphi\,d\nu_{-\infty}\,dt
	+\int_0^T \mathcal{E}_{-\infty}\bigl(\Psi(u)(t),\varphi(t)\bigr)\,dt
	= \int_{K_{-\infty}} u_0\,\varphi(0,\cdot)\,d\nu_{-\infty}
	\]
	for every $\varphi\in C_c^1([0,T);V)$. Moreover, if $u_0\ge0$ $\nu_{-\infty}$-a.e., then $u(t,\cdot)\ge0$ $\nu_{-\infty}$-a.e. for a.e. $t\in(0,T)$.
\end{corollary}

\subsection{A heat kernel construction on metric measure spaces}

In this subsection, we consider a setting where the heat kernel is obtained directly from multiresolution analysis, rather than from a pre-existing Dirichlet form. A construction by Cao, Grigor'yan, and Liu \cite{CaoGrigoryanLiu} introduces such a heat kernel on general metric measure spaces satisfying the volume doubling condition, yielding stochastically complete heat kernels with stable-like upper estimates.

Let $(M,d,\mu)$ be a metric measure space satisfying the volume doubling condition:
\[
\mu(B(x,2r)) \le C_D \mu(B(x,r)), \qquad \forall x\in M,\ r>0,
\]
for some constant $C_D>1$. For every $\beta>0$, the construction in \cite{CaoGrigoryanLiu} yields a stochastically complete positive heat kernel $\{p_t\}_{t>0}$ on $M$ satisfying
\[
\int_M p_t(x,y)\,d\mu(y)=1,\qquad \forall t>0,\ \mu\text{-a.e. }x\in M,
\]
and the stable-like upper estimate
\[
0\le p_t(x,y)\le \frac{C}{V(x,t^{1/\beta}+d(x,y))}
\left(\frac{t^{1/\beta}}{t^{1/\beta}+d(x,y)}\right)^\beta,
\qquad \forall t>0,\ x,y\in M,
\]
where $V(x,r):=\mu(B(x,r))$ and $C$ is a constant independent of $t,x,y$.

Let $\{P_t\}_{t\ge0}$ be the associated semigroup on $L^2(M,\mu)$,
\[
P_t f(x):=\int_M p_t(x,y)f(y)\,d\mu(y).
\]
The associated Laplacian $\mathcal{L}$ of $\{P_t\}_{t\ge0}$ gives rise to a Dirichlet form $(\mathcal{E},\mathcal{F})$ by 
\[
\mathcal{E}(f,f):=\lim_{t\downarrow 0}\frac{1}{t}\langle f-P_t f,f\rangle_{L^2(M,\mu)}=\lim_{t\downarrow 0}\frac{1}{2t}\int_M\int_M p_t(x,y)\left(f(x)-f(y)\right)^2\,d\mu\,d\mu,
\quad f\in L^2(M,\mu),
\]
with domain
\[
\mathcal{F}:=\left\{
f\in L^2(M,\mu): \mathcal{E}(f,f)<\infty
\right\}.
\]
This quadratic form is a Dirichlet form on $L^2(M,\mu)$ (see, e.g., \cite[\S 4.2]{GrigoryanHuLau.2003.TAMS2065}); its associated Laplacian is precisely the operator $\mathcal{L}$ in the sense of Dirichlet form theory.

Set
\[
V:=L^{1+1/m}(M,\mu)\cap \mathcal{F}_e,
\]
where $\mathcal{F}_e$ denotes the extended Dirichlet space of $(\mathcal{E},\mathcal{F})$. Applying Theorem~\ref{thm:main}(i),(ii) with this choice of $V$ yields the following.

\begin{corollary}
	For every $T>0$ and every $u_0\in L^{m+1}(M,\mu)$, there exists a unique
	\[
	u\in L^\infty(0,T;L^{m+1})\cap W^{1,2}(0,T;V^*)
	\]
	such that
	\[
	\Psi(u):=|u|^{m-1}u\in L^2(0,T;V),
	\]
	and $u$ satisfies
	\[
	-\int_0^T\int_M u\,\partial_t\varphi\,d\mu\,dt
	+\int_0^T \mathcal{E}(\Psi(u)(t),\varphi(t))\,dt
	= \int_M u_0\,\varphi(0,\cdot)\,d\mu
	\]
	for every $\varphi\in C_c^1([0,T);V)$. Moreover, if $u_0\ge0$ $\mu$-a.e., then $u(t,\cdot)\ge0$ $\mu$-a.e. for a.e. $t\in(0,T)$.
\end{corollary}

\appendix

\section{Proof of Lemma~\protect{\ref{lem:Fe_closed}}}
\label{App:A}

In this section, we prove Lemma~\ref{lem:Fe_closed}. The argument relies on the following diagonal extraction lemma.

\begin{proposition}[Diagonal extraction lemma]
	\label{diagonal}
	Let $\mu$ be a $\sigma$-finite measure on a measurable space $X$, and let $u:X\to\mathbb{R}$ be measurable. Suppose that $\{u_n\}_{n\ge1}$ is a sequence of measurable functions such that $u_n\to u$ $\mu$-a.e. as $n\to\infty$, and that for each $n\ge1$, there exists a sequence $\{v_{n,k}\}_{k\ge1}$ of measurable functions such that $v_{n,k}\to u_n$ $\mu$-a.e. as $k\to\infty$. Then there exists a map $g:\mathbb{N}\to\mathbb{N}$ with $g(n)\to\infty$ as $n\to\infty$ such that
	\[
	v_{n,g(n)}\to u \quad \mu\text{-a.e. as } n\to\infty.
	\]
	Moreover, if $h:\mathbb{N}\to\mathbb{N}$ satisfies $h(n)\ge g(n)$ for all $n$, then $v_{n,h(n)}\to u$ $\mu$-a.e. as well.
\end{proposition}

\begin{proof}
	Since $\mu$ is $\sigma$-finite, there exists an increasing sequence of measurable sets $\{X_m\}_{m\ge1}$ such that $X_m\uparrow X$ and $\mu(X_m)<\infty$ for each $m\ge1$.
	
	For each $m\ge1$ and $n\ge1$, by Egorov's theorem applied on the finite-measure set $X_m$, there exist measurable sets $A_m$ and $B_{m,n}$ such that
	\[
	\mu(A_m)<2^{-m},\qquad \mu(B_{m,n})<2^{-(m+n)},
	\]
	and
	\[
	u_n\to u \text{ uniformly on } X_m\setminus A_m,
	\qquad
	v_{n,k}\to u_n \text{ uniformly on } X_m\setminus B_{m,n}.
	\]
	Define the good sets
	\[
	C_m:=X_m\setminus\left(A_m\cup\bigcup_{n=1}^\infty B_{m,n}\right),\qquad m\ge1.
	\]
	On each $C_m$, we have uniform convergence $u_n\to u$, and for each fixed $n$, uniform convergence $v_{n,k}\to u_n$ as $k\to\infty$.
	
	Set
	\[
	C:=\bigcup_{n=1}^\infty \bigcap_{m=n}^\infty C_m.
	\]
	We claim that $\mu(X\setminus C)=0$. Indeed,
	\[
	X\setminus C\subset \bigcap_{n=1}^\infty \bigcup_{m=n}^\infty (X_m\setminus C_m),
	\]
	and
	\[
	\sum_{m=1}^\infty \mu(X_m\setminus C_m)
	\le \sum_{m=1}^\infty \left(\mu(A_m)+\sum_{n=1}^\infty \mu(B_{m,n})\right)
	\le \sum_{m=1}^\infty \left(2^{-m}+\sum_{n=1}^\infty 2^{-(m+n)}\right)
	= \sum_{m=1}^\infty 2^{-(m-1)}
	=2<\infty.
	\]
	Thus $\mu(X\setminus C)=0$ by the Borel--Cantelli lemma.
	
	We now construct $g$. For each fixed $n$, since $v_{n,k}\to u_n$ uniformly on each $C_m$, and since $C$ is covered by the nested sequence $\{\bigcap_{m=n}^\infty C_m\}_{n\ge1}$, for every $x\in C$ we may choose $K(n,m,l)$ such that
	\[
	\sup_{y\in C_m}|v_{n,k}(y)-u_n(y)|<\frac1l \quad\text{whenever } k\ge K(n,m,l).
	\]
	Also, since $u_n\to u$ uniformly on each $C_m$, for each $l$ there exists $n(m,l)$ such that
	\[
	\sup_{y\in C_m}|u_n(y)-u(y)|<\frac1l \quad\text{whenever } n\ge n(m,l).
	\]
	
	Define
	\[
	g(n):=\max_{1\le l\le n} K(n,l,l).
	\]
	We show that $v_{n,g(n)}\to u$ pointwise on $C$. Fix $x\in C$ and $\varepsilon>0$. Since $x\in C$, there exists $a:=a(x)\in\mathbb{N}_+$ such that $x\in C_m$ for all $m\ge a$. Choose $l\in\mathbb{N}$ such that $l\ge a$ and $1/l<\varepsilon/2$. Then, for any $n\ge l\vee n(l,l)$, we have
	\[
	|u_n(x)-u(x)|\le \sup_{y\in C_l}|u_n(y)-u(y)|<\frac1l<\frac{\varepsilon}{2},
	\]
	and, since $g(n)\ge K(n,l,l)$,
	\[
	|v_{n,g(n)}(x)-u_n(x)|\le \sup_{y\in C_l}|v_{n,g(n)}(y)-u_n(y)|<\frac1l<\frac{\varepsilon}{2}.
	\]
	Therefore $|v_{n,g(n)}(x)-u(x)|<\varepsilon$ for all sufficiently large $n\ge l\vee n(l,l)$. Thus $v_{n,g(n)}\to u$ pointwise on $C$, hence $\mu$-a.e.
	
	The final assertion for any $h\ge g$ follows by the same argument, since $h(n)\ge g(n)\ge K(n,l,l)$ for all sufficiently large $n$.
\end{proof}

We now prove Lemma~\ref{lem:Fe_closed} using the diagonal extraction lemma.

\begin{proof}[Proof of Lemma~\ref{lem:Fe_closed}]
	Let $\{v_n\}_{n\ge1}\subset\mathcal{F}_e$ be an $\mathcal{E}$-Cauchy sequence such that $v_n\to v$ $\mu$-a.e. For each $n\ge1$, by definition of $\mathcal{F}_e$, there exists an $\mathcal{E}$-Cauchy sequence $\{v_{n,k}\}_{k\ge1}\subset\mathcal{F}$ such that $v_{n,k}\to v_n$ $\mu$-a.e. as $k\to\infty$.
	
	We recall that, under our standing assumption that $(M,d)$ is locally compact and separable and $\mu$ is a Radon measure, $\mu$ is $\sigma$-finite. This justifies the use of the diagonal extraction argument below.
	
	For every $n\ge1$, applying Proposition~\ref{diagonal} to the sequence $\{v_{n,k}\}$ yields a map $g:\mathbb{N}\to\mathbb{N}$ with $g(n)\to\infty$ such that
	\[
	v_{n,g(n)}\to v \quad \mu\text{-a.e. as } n\to\infty.
	\]
	On the other hand, since for each fixed $n$ we have $v_{n,k}\to v_n$ in the $\mathcal{E}$-sense as $k\to\infty$ by definition of $\mathcal{F}_e$, there exists a map $h:\mathbb{N}\to\mathbb{N}$ such that
	\[
	\mathcal{E}(v_{n,h(n)}-v_n)<\frac1{n^2}\quad\text{for all } n\ge1.
	\]
	Now define
	\[
	f(n):=\max\{g(n),h(n)\}\qquad(n\ge1).
	\]
	Then $f(n)\ge g(n)$ for all $n$, so by the final assertion of Proposition~\ref{diagonal}, we still have
	\[
	v_{n,f(n)}\to v \quad \mu\text{-a.e. as } n\to\infty.
	\]
	Moreover, since $f(n)\ge h(n)$ for all $n$, we also have
	\[
	\mathcal{E}(v_{n,f(n)}-v_n)<\frac1{n^2}\quad\text{for all } n\ge1.
	\]
	Thus $f$ simultaneously controls both the $\mu$-a.e. convergence and the energy error.
	
	Finally, by the triangle inequality for $\sqrt{\mathcal{E}}$,
	\[
	\sqrt{\mathcal{E}(v_n-v)}
	\le \sqrt{\mathcal{E}(v_n-v_{n,f(n)})}+\sqrt{\mathcal{E}(v_{n,f(n)}-v)}
	\le \frac1n+\sqrt{\mathcal{E}(v_{n,f(n)}-v)}.
	\]
	Letting $n\to\infty$, the right-hand side tends to $0$; hence $\mathcal{E}(v_n-v)\to0$ as $n\to\infty$. This proves $v\in\mathcal{F}_e$ and the desired convergence. The proof is complete.
\end{proof}

\begin{acknowledgement}
	The author thanks Meng Yang for valuable discussions on the extended Dirichlet space, which were instrumental in developing the functional framework of this work.
\end{acknowledgement}

\bibliographystyle{siam}
\bibliography{ref}

\end{document}